\documentclass[11pt,reqno]{amsart}

\usepackage{amsfonts}
\usepackage{amssymb}
\usepackage{amsmath}
\usepackage{amsthm} 
\usepackage{listings}
\usepackage{algorithmic}
\usepackage{algorithm}
\usepackage[colorlinks=true]{hyperref}

\usepackage{tikz}
\usetikzlibrary{matrix,arrows,patterns}

\usepackage{multirow}
\usepackage{bigstrut}

\usepackage{rotating}

\newcommand{\Po}{\mathbb{P}}	 
\DeclareMathOperator{\rem}{rem}	

\newcommand{\uth}{^{\text{th}}}
\newcommand{\rmr}{\mathrm{r}}
\newcommand{\rmc}{\mathrm{c}}
\newcommand{\rmi}{\mathrm{i}}
\newcommand{\rmd}{\mathrm{d}}

\DeclareMathOperator{\invsum}{\mathsf{invsum}}
\DeclareMathOperator{\ninv}{\mathsf{ninv}}
\DeclareMathOperator{\inv}{\mathsf{inv}}

\DeclareMathOperator{\des}{\mathsf{des}}
\DeclareMathOperator{\modinv}{\mathsf{modinv}}
\DeclareMathOperator{\lbsum}{\mathsf{lbsum}}
\DeclareMathOperator{\ipcni}{\mathsf{ipcni}}
\DeclareMathOperator{\ninvsum}{\mathsf{ninvsum}}
\DeclareMathOperator{\NINV}{\mathsf{NINV}}
\DeclareMathOperator{\INV}{\mathsf{INV}}
\DeclareMathOperator{\zcv}{\mathsf{nzcv}}
\DeclareMathOperator{\zicv}{\mathsf{izcv}}
\DeclareMathOperator{\izcv}{\mathsf{izcv}}
\DeclareMathOperator{\nzcv}{\mathsf{nzcv}}
\DeclareMathOperator{\aizcv}{\mathsf{aizcv}}
\DeclareMathOperator{\anzcv}{\mathsf{anzcv}}
\newcommand{\symS}{{\mathfrak S}}					 	
\DeclareMathOperator{\boks}{box}

\newcommand{\eqdef}{\stackrel{\mbox{\tiny def}}{=}} 

\newcommand{\pattern}[4]{										
  \raisebox{0.6ex}{
  \begin{tikzpicture}[scale=0.35, baseline=(current bounding box.center), #1]
  \useasboundingbox (0.0,-0.1) rectangle (#2+1.4,#2+1.1);
    \foreach \x/\y in {#4}
      \fill[pattern=north east lines] (\x,\y) rectangle +(1,1);
    \draw (0.01,0.01) grid (#2+0.99,#2+0.99);
    \foreach \x/\y in {#3}
      \filldraw (\x,\y) circle (6pt);
  \end{tikzpicture}}
}

\newcommand{\patternsbmm}[6]{									
  \raisebox{0.6ex}{											
  \begin{tikzpicture}[scale=0.35, baseline=(current bounding box.center), #1]
  \useasboundingbox (0.0,-0.1) rectangle (#2+1.4,#2+1.1);
    \foreach \x/\y in {#4}
      \fill[pattern=north east lines] (\x,\y) rectangle +(1,1);
    \draw (0.01,0.01) grid (#2+0.99,#2+0.99);
    \foreach \x/\y/\z/\w/\A in {#5}
       {
       \fill[color = white!100, opacity=1, rounded corners] (\x+0.075,\y+0.075) rectangle (\z-0.075,\w-0.075);
       \draw[color = black, rounded corners] (\x+0.075,\y+0.075) rectangle (\z-0.075,\w-0.075);
       }
    \foreach \x/\y/\z/\w/\A in {#6}
       \fill[black] (\x/2+\z/2,\y/2+\w/2) node {$\scriptstyle\A$};
    \foreach \x/\y in {#3}
      \filldraw (\x,\y) circle (6pt);

  \end{tikzpicture}}
}

\newcommand{\qchoose}[2]{
\left[\!\!\begin{array}{c}#1 \\#2\end{array}\!\!\right]_q
}

\theoremstyle{plain}
\newtheorem{theorem}{Theorem}[section]
\newtheorem{proposition}[theorem]{Proposition}
\newtheorem{lemma}[theorem]{Lemma}
\newtheorem{corollary}[theorem]{Corollary}

\theoremstyle{definition}
\newtheorem{definition}[theorem]{Definition}
\newtheorem{example}[theorem]{Example}

\newcommand{\pairsWsWl}[8]{
  \raisebox{0.6ex}{
  \begin{tikzpicture}[scale=0.35, baseline=(current bounding box.center), #1]
    \foreach \x in {#2}
      \fill (\x*#7,0) node {$\x$};
    \foreach \x in {#3}
    	   \draw[line width = 1pt]
       (\x*#7,0.5) -- (\x*#7,1.5) -- (\x*#7+#8*#7,1.5) -- (\x*#7+#8*#7,0.5) ;
    \foreach \x in {#4}
    	   \draw[line width = 1pt]
       (\x*#7,0.5) -- (\x*#7,2.5) -- (\x*#7+#8*#7,2.5) -- (\x*#7+#8*#7,0.5) ;
    \foreach \x in {#5}
    	   \draw[line width = 1pt]
       (\x*#7,-0.5) -- (\x*#7,-1.5) -- (\x*#7+#8*#7,-1.5) -- (\x*#7+#8*#7,-0.5) ;
    \foreach \x in {#6}
    	   \draw[line width = 1pt]
       (\x*#7,-0.5) -- (\x*#7,-2.5) -- (\x*#7+#8*#7,-2.5) -- (\x*#7+#8*#7,-0.5) ;
  \end{tikzpicture}}
}

\begin{document}

\title{Refined inversion statistics  on permutations}
\author[Sack]{Joshua Sack}
\author[\'Ulfarsson]{Henning \'Ulfarsson}

\thanks{Sack was partly supported by grant no.\ 100048021 from the Icelandic Research Fund.\\
\'Ulfarsson was supported by grant no.\ 090038011 from the Icelandic Research Fund.}

\address[Sack]{Department of Mathematics and Statistics, California State University, Long Beach, USA}
\address[\'Ulfarsson]{School of Computer Science, Reykjavik University, Iceland}

\email{joshua.sack@gmail.com, henningu@ru.is}

\date{Updated: \today}

\begin{abstract}
We introduce and study new refinements of inversion statistics for permutations, such as $k$-step inversions, (the number of inversions with fixed position differences) and non-inversion sums (the sum of the differences of positions of the non-inversions of a permutation). 
We also provide a distribution function for non-inversion sums, a distribution function for $k$-step inversions that relates to the Eulerian polynomials, and special cases of distribution functions for other statistics we introduce, such as $(\le\!\!k)$-step inversions and $(k_1,k_2)$-step inversions (that fix the value separation as well as the position).  We connect our refinements to other work, such as inversion tops that are $0$ modulo a fixed integer $d$, left boundary sums of paths, and marked meshed patterns.   Finally, we use non-inversion sums to show that for every number $n>34$, there is a permutation such that the dot product of that permutation and the identity permutation (of the same length) is $n$.
\end{abstract}

\maketitle

\setcounter{tocdepth}{1}
\tableofcontents

\thispagestyle{empty}

\section{Introduction}

The main object of study in this paper is the set of inversions in a permutation.\footnote{We provide basic definitions at the end of this introduction.} An \emph{inversion} in a permutation $\pi$, of rank $n$, is a pair $(a,b)$ satisfying $1 \leq a < b \leq n$ and $\pi(a) > \pi(b)$. All other pairs are called \emph{non-inversions}. We are particularly interested in permutation statistics related to inversions, such as the number of inversions of a certain form. The study of permutation statistics was largely initiated by the seminal MacMahon~\cite{MR2417935}, but has seen explosive growth in recent decades. In Section~\ref{sec:ninvsums-and-dotp} we introduce the concept of the \emph{non-inversion sum} of a permutation. This is the sum of the differences $b-a$ for all non-inversions $(a,b)$ in the permutation. Before studying the distribution of this statistic we connect these non-inversion sums to another known statistic on permutations: the dot product with a fixed vector.
In particular, the dot product of the permutation (treated as a vector) with the identity permutation of the same length is equal to the non-inversion sum of the permutation plus a function of the rank of the permutation; see Theorem~\ref{thm:1dotpi}.

In Section~\ref{sec:zcvs-and-ninvsumdist}, we define the distribution function for the non-inversion sum and prove a recurrence relation for it in Theorem~\ref{thm:ninvsumdist}. We introduce the concept of a \emph{zone-crossing vector}, which appears in the recurrence relations. This is a vector whose $k\uth$ coordinate is the number of non-inversions $(a,b)$ such that $a \leq k < b$. We relate these vectors to the non-inversion sums and show that there is a bijective correspondence between permutations and their zone-crossing vectors. We also prove a theorem showing that the distribution of the coordinates of these vectors is related to the $q$-analog of the binomial coefficients; see Theorem~\ref{thm:zcvdist}.

In Section~\ref{subsec:Hnk} we consider $k$-step inversions, which are inversions $(a,b)$ such that $b-a = k$, and show in Theorem~\ref{thm:H_nk} that the distribution of these types of inversions is related to the Eulerian polynomials. 
We next consider $(k_1,k_2)$-step inversions, which are inversions $(a,b)$, such that $b-a = k_1$ and $\pi(b)-\pi(a)=k_2$, and prove a special case of the distribution function; see Proposition~\ref{proposition:k1k2specialcase}.
We also consider inversions $(a,b)$ such that $b-a \leq k$ and prove recurrence relations for their distributions in some special cases; see Proposition~\ref{prop:leq_k-step:n-1case}.

In Section~\ref{sec:certified-and-modulo}, we consider some relationships between our work and the work of others.
In Section~\ref{section:topsModd}, we consider a $k$-step variant of a statistic that counts inversions whose first coordinate (called the inversion top) is 0 modulo $d$.
Inversion tops $\bmod d$ have been studied by Kitaev and Remmel~\cite{MR2240770,MR2336014} and by Jansson~\cite{Jansson}.
We provide formulas for special cases of the distribution of $k$-step inversions whose first coordinate is $0$ mod $d$. 

In Section~\ref{section:ipcni}, we consider a $k$-step variant of the left boundary sums in Dukes and Reifergerste~\cite{MR2628782}.
Given a permutation $\pi$, the left boundary sum of $\pi$ (denoted $\lbsum(\pi)$) gives the area to the left of the Dyck path of $\pi$.
Dukes and Reifergerste~\cite{MR2628782} show that $\lbsum(\pi)$ is also the sum of the number of inversions and the number of certified non-inversions, where a certified non-inversion is a non-inversion $(a,b)$, with a position $c$, such that $a<c<b$ and $\pi_c\ge \pi_d$ whenever $a < d <b$.
We consider a $k$-step variant of this (denoted $\ipcni_k(\pi)$) that only counts $k$-step inversions and $k$-step certified non-inversions, and provide special cases of the distribution functon.
Finally we show how many of the statistics we consider can be represented using marked mesh patterns defined by \'Ulfarsson in~\cite{U11}

The connection found in Theorem~\ref{thm:1dotpi} is used in Theorem~\ref{thm:existPIforCOS} to show that given any integer $k$ greater than $34$ there exists a permutation $\pi$ such that the dot product of $\pi$ with the identity permutation $12\dotsm|\pi|$ equals $k$. We also present an algorithm that, given $k$, produces the permutation $\pi$; see Section~\ref{section:algorithm}.
The total number of permutations which dotted with the identity permutation gives $k$, is given by the sequence A135298\footnote{http://oeis.org/A135298} in the Online Encyclopedia of Integer Sequences, and hence our theorem tells us that this sequence is non-zero after $k = 34$.

\subsubsection*{Basic definitions}
We define the set of positive integers to be $\Po = \{1,2,3,\dotsc\}$.
A \emph{permutation} is a bijective function $\pi:\{1,\ldots,n\}\to \{1,\ldots,n\}$ for some $n$ in $\Po$.
The number $n$ is called the \emph{rank} of the permutation.
We often write $\pi_k$ for $\pi(k)$, and write a permutation as a list of its values $\pi_1\pi_2\dotsm\pi_n$.
Let $\symS_n$ be the set of permutations of rank $n$.

We define the identity permutation $\mathbf{1}_n$ as the permutation $\pi$, such that $\pi_k = k$ for $1\le k \le n$. We will write $\mathbf{1}$, omitting the subscript, if the rank is clear from the context.
Given a permutation $\pi = \pi_1\pi_2\cdots\pi_n$, we define its \emph{reverse} as $\pi^\rmr = \pi_n\pi_{n-1}\dotsm\pi_1$, its \emph{complement} as $\pi^\rmc = (n+1-\pi_1)(n+1-\pi_2)\dotsm(n+1-\pi_n)$, and its \emph{inverse} $\pi^\rmi$ as the unique permutation such that $\pi \circ \pi^\rmi = \mathbf{1}$.


\section{Non-inversion sums and the dot product of permutations} \label{sec:ninvsums-and-dotp}


\begin{definition}
For a permutation $\pi$ of rank $n$, the number
\[
\mathbf{1}\cdot \pi = \sum_{i=1}^{n} i\pi(i)
\]
is called the \emph{cosine} of the permutation.
\end{definition}
Note that if we treat permutations as vectors then
\[
\mathbf{1} \cdot \pi = |\mathbf{1}| \cdot |\pi| \cos(\theta) =
(1^2 + 2^2 + \dotsb + n^2) \cos(\theta)
=\frac{n(n+1)(2n+1)}{6} \cos(\theta),
\]
where $\theta$ is the angle between $\mathbf{1}$ and $\pi$. So $\mathbf{1} \cdot \pi$ only depends on the cosine of the angle between the identity and the
permutation.

Most of this section will be leading to a proof of the following theorem:
\begin{theorem} \label{thm:existPIforCOS}
For a positive integer 
\[
k\not\in\{2,3,6,7,8,9,12,15,16,17,18,19,31,32,33,34\},
\]
there exists a permutation $\pi$ such that $\mathbf{1}\cdot\pi = k$.
\end{theorem}
The total number of permutations $\pi$, such that $\mathbf{1}\cdot\pi= k$, is given by the sequence A135298\footnote{http://oeis.org/A135298} in the Online Encyclopedia of Integer Sequences. 
Our theorem tells us that this sequence is non-zero after $k = 34$.
Furthermore, we will provide an algorithm in Section~\ref{section:algorithm} for constructing a permutation $\pi$, such that $\mathbf{1}\cdot\pi = k$ for $k$ as in the theorem.

To prove this theorem, we introduce the notion of the non-inversion sum.
We build this notion on that of a non-inversion.
Given a permutation $\pi$ of rank $n$, an \emph{inversion} is a pair $(a,b)$, such that $1\le a<b\le n$ and $\pi(a) > \pi(b)$, and a \emph{non-inversion} is a pair $(a,b)$, such that $1\le a<b\le n$ and $\pi(a) < \pi(b)$.
Denote the set of inversions of $\pi$ by $\INV(\pi)$, and the set of non-inversions by $\NINV(\pi)$.
\begin{definition} 
Let $\pi$ be a permutation.
\begin{enumerate}
\item The number
\[
\invsum(\pi) = \sum_{(a,b)\in \INV(\pi)}\left(b-a\right),
\]
is called the \emph{inversion sum} of $\pi$.
\item The number
\[
\ninvsum(\pi) = \sum_{(a,b)\in \NINV(\pi)}\left(b-a\right),
\]
is called the \emph{non-inversion sum} of $\pi$.
\end{enumerate}
\end{definition}
Observe that the values added up in the sums are differences of positions ($b-a$)
rather than of values ($\pi(b) - \pi(a)$). The following result shows that had we
defined the sums in terms of differences of values we would have resulted in the
same function.
\begin{proposition} \label{lem:ninvpininvpii}
For any permutation $\pi$
\[
\ninvsum(\pi^\rmi) = \ninvsum(\pi),
\]
or equivalently
\[
\sum_{(a,b) \in \NINV(\pi)} \left(\pi(b)-\pi(a)\right) = \sum_{(a,b) \in \NINV(\pi)} \left(b-a\right).
\]
A similar statement holds for the inversion sum.
\end{proposition}

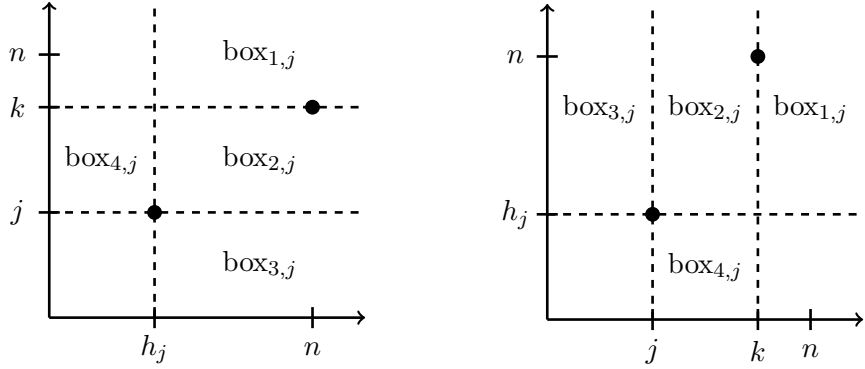
\begin{figure}[h]
\begin{center}
\begin{tikzpicture}[scale=1.4]

    \draw[line width = 1pt, ->] (0,0) -- (3,0);
    \draw[line width = 1pt, ->] (0,0) -- (0,3);
    
    \draw[line width = 1pt] (-0.1,1.0) -- (0.1,1.0);
    \fill (-0.3,1.0) node {$j$};
    \draw[line width = 1pt, style = {dashed}] (0,1) -- (3,1);
    
    \draw[line width = 1pt] (1,-0.1) -- (1,0.1);
    \fill (1,-0.3) node {$h_j$};
    \draw[line width = 1pt, style = {dashed}] (1,0) -- (1,3);
    \fill (1,1) circle (2pt);
    
    \draw[line width = 1pt] (-0.1,2.5) -- (0.1,2.5);
    \fill (-0.3,2.5) node {$n$};
        
    \draw[line width = 1pt] (-0.1,2) -- (0.1,2);
    \fill (-0.3,2) node {$k$};
    
    \draw[line width = 1pt] (2.5,-0.1) -- (2.5,0.1);
    \fill (2.5,-0.3) node {$n$};
    \fill (2.5,2) circle (2pt);
    
    \draw[line width = 1pt, style = {dashed}] (0,2) -- (3,2);
    
    \fill (2,2.5) node {${\boks_{1,j}}$};
    \fill (2,1.5) node {${\boks_{2,j}}$};
    \fill (2,0.5) node {${\boks_{3,j}}$};
    \fill (0.5,1.5) node {$\boks_{4,j}$};
  
\end{tikzpicture}
\qquad\qquad
\begin{tikzpicture}[scale=1.4]

    \draw[line width = 1pt, ->] (0,0) -- (3,0);
    \draw[line width = 1pt, ->] (0,0) -- (0,3);

    \draw[line width = 1pt] (-0.1,2.5) -- (0.1,2.5);
    \fill (-0.3,2.5) node {$n$};
    \fill (2,2.5) circle (2pt);
    
    \draw[line width = 1pt] (1,-0.1) -- (1,0.1);
    \fill (1,-0.3) node {$j$};
    \fill (1,1) circle (2pt);
    \draw[line width = 1pt, style = {dashed}] (1,0) -- (1,3);
    
    \draw[line width = 1pt] (-0.1,1) -- (0.1,1);
    \fill (-0.3,1) node {$h_j$};
    
    \draw[line width = 1pt] (2,-0.1) -- (2,0.1);
    \fill (2,-0.3) node {$k$};
    
    \draw[line width = 1pt] (2.5,-0.1) -- (2.5,0.1);
    \fill (2.5,-0.3) node {$n$};
    
    \draw[line width = 1pt, style = {dashed}] (2,0) -- (2,3);
    \draw[line width = 1pt, style = {dashed}] (0,1) -- (3,1);
    
    \fill (2.5,2) node {${\boks_{1,j}}$};
    \fill (1.5,2) node {${\boks_{2,j}}$};
    \fill (0.5,2) node {$\boks_{3,j}$};
    \fill (1.5,0.5) node {${\boks_{4,j}}$};
  
\end{tikzpicture}

\caption{The permutation $\pi$ is shown on the left and $\pi^\rmi$ is shown on the right.}
\label{fig:boxes}
\end{center}
\end{figure}

\begin{proof} 
We will prove the statement by induction on the rank of the permutation. Let $\pi$ be an arbitrary permutation and let $\pi(n) = k$. If $k = 1$
then the result follows immediately by the induction hypotheses. Otherwise let $\pi(h_j) = j$ for $j = 1, \dotsc, k-1$.
We depict in Figure~\ref{fig:boxes} graphs of $\pi$ and $\pi^\rmi$, where the $\boks_{i,j}$ represents the sets of pairs $(a,\pi_a)$ lying in the designated regions of the graph on the left, or $(a, \pi^\rmi_a)$ lying in the designated regions of the graph on the right.
For example, $\boks_{2,j} = \{(a,\pi_a)\mid h_j<a,\, j<\pi_a<k\}$.
Let $\tau$ be the permutation obtained from $\pi$ by removing the last element $k = \pi(n)$ and reducing the letters of $\pi$ that are larger than $k$ by $1$.
Then, by the induction hypothesis, $\ninvsum(\tau) = \ninvsum(\tau^\rmi)$. But
\[
\ninvsum(\pi) = \ninvsum(\tau) + \sum_{j=1}^{k-1} 1 + |\boks_{1,j}| + |\boks_{2,j}| + |\boks_{3,j}|,
\]
where for each $j$ the sum of the box sizes is equal to one less than the separation $n-h_j$,
and
\[
\ninvsum(\pi^\rmi) = \ninvsum(\tau^\rmi) + \sum_{j=1}^{k-1} 1 + |\boks_{1,j}| + |\boks_{2,j}| + |\boks_{4,j}|,
\]
where for each $j$ the sum of the box sizes $\boks_{2,j}$ and $\boks_{4,j}$ is equal to one less than the separation
$k-j$ and the size of $\boks_{1,j}$ represents the number of former non-inversions whose separation
has just increased by one.

\noindent
To see that $\sum_{j=1}^{k-1} |\boks_{3,j}|$ is equal to $\sum_{j=1}^{k-1} |\boks_{4,j}|$ note
that the following are equivalent:
\begin{itemize}
\item $(a,\sigma(a))\in \boks_{4,\pi(b)}$,
\item $(a,b)\in \INV(\pi)$ with $\pi(a)<k$,
\item $(b,\sigma(b))\in \boks_{3,\pi(a)}$. \qedhere
\end{itemize}
\end{proof}
It is straightforward to see that $\ninvsum(\pi^\rmr)  = \invsum(\pi) = \ninvsum(\pi^\rmc)$.

Note that for any permutation $\pi$ of rank $n$, the sum of the inversion sum and the non-inversion sum is the $(n-1)\uth$ tetrahedral number $\binom{n+1}{3}$:
\begin{align} \label{eq:invs-ninvs}
\invsum(\pi) + \ninvsum(\pi) &= \sum_{1 \leq a<b \leq n} (b-a) \\
&= \frac{(n-1)n(n+1)}{6} = \binom{n+1}{3}, \nonumber
\end{align}
so two permutations have the same inversion sum if and only if they have
the same non-inversion sum.

We now show that the cosine of the permutation is closely related to the non-inversion
sum of the permutation.

\begin{theorem} \label{thm:1dotpi}
For any permutation $\pi$,
\[
\mathbf{1}\cdot \pi = \mathbf{1}\cdot \mathbf{1}^\rmc + \ninvsum(\pi).
\] 
\end{theorem}
\begin{proof}
Let $\varphi$ be a function mapping a permutation $\pi$ of rank $n$ to a vector, whose $j\uth$ coordinate is the number of times the $j\uth$ position of $\pi$ is at the end of a non-inversion minus the number of times the $j\uth$ position is at the beginning of a non-inversion, that is, 
\[
\varphi(\pi)_j = \sum_{(i,j)\in \NINV(\pi)}1 - \sum_{(j,k)\in \NINV(\pi)}1.
\]
The $j\uth$ coordinate of $\varphi(\pi)$ is then the coefficient of $j$ (treating $j$ as a variable) in the non-inversion sum formula, and hence the contribution of the $j\uth$ position of $\pi$ to the non-inversion sum is $j$ times this number.
Thus $\ninvsum(\pi) = \mathbf{1}\cdot \varphi(\pi)$.

We next see that the $j\uth$ coordinate of $\varphi(\pi)$ is $\varphi(\pi)_j = \pi_j - \mathbf{1}^{\rmc}_j$.
The first coordinate is $\varphi(\pi)_1 = \pi_1 - n = \pi_1 - \mathbf{1}^{\rmc}_1$, since in the formula for the non-inversion sum, $\pi_1$ will be subtracted once for every non-inversion, which is guaranteed by a value greater than $\pi_1$.
For general $j \geq 1$, if $\pi_{j}-\pi_{j+1} >0$, then $\varphi(\pi_{j+1})$ can be obtained from $\varphi(\pi_j)$ by subtracting the number of values between $\pi_{j+1}$ and $\pi_j$, as given each such value $\pi_k$, either $k<j$, in which case $(k,j)$ was counted positively toward $\varphi(\pi_j)$ but $(k,j+1)$ does not count toward $\varphi(\pi_{j+1})$, or $j>j+1$, in which case $(j,k)$ did not count toward $\varphi(\pi_j)$, but $(j+1,k)$ counts negatively toward $\varphi(\pi_{j+1})$.
Thus we subtract $\pi_j -\pi_{j+1} -1$.
If $\pi_j-\pi_{j+1}<0$, then to obtain $\varphi(\pi_j)$ we add 1 for every value between $\pi_{j+1}$ and $\pi_{j}$, and we add 2 in order to account for the non-inversion $(j,j+1)$.
Thus we add $\pi_{j+1} - \pi_j - 1 + 2$.
Either way, we obtain the formula:
\[
\varphi(\pi)_{j+1} = \varphi(\pi)_j + \pi_{j+1} - \pi_{j} -1.
\]
By induction, let us assume that $\varphi(\pi)_j = \pi_j - \mathbf{1}^{\rmc}_j$.
Thus
\[
\varphi(\pi)_{j+1} = \pi_j-\mathbf{1}^{\rmc}_j +\pi_{j+1} - \pi_j - 1 = \pi_{j+1} - \mathbf{1}^{\rmc}_j-1 = \pi_{j+1} - \mathbf{1}^{\rmc}_{j+1} .
\]

In conclusion:
\[
\ninvsum(\pi) = \mathbf{1}\cdot \varphi(\pi) = \mathbf{1}\cdot(\pi-\mathbf{1}^{\rmc}) = \mathbf{1}\cdot\pi - \mathbf{1}\cdot\mathbf{1}^{\rmc},
\]
whence our desired result of this theorem immediately follows.
\end{proof}

Note that for $\mathbf{1} \in \symS_n$, $\mathbf{1} \cdot \mathbf{1}^\rmc = \binom{n+2}{3}$, so equation~\ref{eq:invs-ninvs} implies
that the equation in the theorem is equivalent to
\[
\mathbf{1} \cdot \pi = \binom{n+2}{3} + \binom{n+1}{3} - \invsum(\pi),
\]
which can be simplified to
\[
\mathbf{1} \cdot \pi = \frac{n(n+1)(2n+1)}{6} - \invsum(\pi).
\]

\begin{corollary}
Given two permutations $\pi,\rho\in \symS_n$, 
\[
\ninvsum(\pi\circ \rho) = \pi\cdot \rho^\rmi - \mathbf{1}\cdot\mathbf{1}^\rmc.
\]
\end{corollary}
\begin{proof}
By a direct calculation,
\[
\ninvsum(\pi \circ\rho) = \mathbf{1}\cdot(\pi\circ \rho) - \mathbf{1}\cdot\mathbf{1}^\rmc =  \pi\cdot \rho^\rmi - \mathbf{1}\cdot\mathbf{1}^\rmc.\hfill\qedhere
\]
\end{proof}
Observe that since $\pi \cdot \rho = \rho \cdot \pi$, then $\ninvsum(\pi\circ\rho^\rmi) = \ninvsum(\rho\circ \pi^\rmi)$.
Then taking $\rho= \mathbf{1}$, we get $\ninvsum(\pi) = \ninvsum(\pi^\rmi)$.
This serves as an alternative proof to Proposition~\ref{lem:ninvpininvpii}.

\begin{lemma}\label{lemma:pascal}
For $n \ge 6$, 
\[
\binom{n+1}{3} + \binom{n}{3} \ge \binom{n+2}{3} - 1.
\]
\end{lemma}
\begin{proof}
A straightforward calculation shows that for $n\ge 7$, $\binom{n+1}{3} + \binom{n}{3} > \binom{n+2}{3}$.
For the case where $n = 6$, note that $\binom{7}{3} + \binom{6}{3} = \binom{8}{3} - 1$.
\end{proof}

\begin{lemma}\label{lemma:nogaps4}
For each value $0\le k\le 10$, there exists a permutation $\pi\in \symS_4$, such that $\ninvsum(\pi) = k$.
\end{lemma}
\begin{proof}
Here is a permutation for each value of $k$: $4321$, $3421$, $3412$, $4213$, $4123$, $2413$, $3214$, $1423$, $2143$, $1243$, $1234$.
\end{proof}
\begin{lemma}\label{lemma:nogaps}
For $n \ge 4$ and each $0\le k \le \binom{n+1}{3}$, there is a permutation $\pi\in \symS_n$, such that $\ninvsum(\pi) = k$.
\end{lemma}
\begin{proof}
We show this by induction on $n$, where the base case ($n=4$) is given by Lemma~\ref{lemma:nogaps4}.
Assuming this holds for $n-1$ (with $n>4$), we consider permutations $\pi\in \symS_n$, with $\pi_n =1$.
The last entry does not contribute anything to the non-inversion sum of the first $n-1$, which by the induction hypothesis ranges through all the integers in the interval from $0$ through $\binom{n}{3}$.
Next, consider permutations $\pi\in \symS_n$, with $\pi_1=1$.
This first entry is guaranteed to contribute $\binom{n}{2}$ to the non-inversion sum, while the rest can be chosen to contribute any integer ranging from $0$ through $\binom{n}{3}$.
Because $\binom{n+1}{3} = \binom{n}{3} + \binom{n}{2}$, and because $\binom{n}{3} >\binom{n}{2}$ for $n>3$, we have that we can obtain every integer from $0$ through $\binom{n+1}{3}$.
\end{proof}

We are now ready to prove the main theorem of this section.
\begin{proof}[Proof of Theorem~\ref{thm:existPIforCOS}]
Given $k \ge 35$, let $n$ be the largest integer, such that $\binom{n+2}{3} \le k $. 
Note that $n\ge 5$.
Let $m = k - \binom{n+2}{3}$.
For $n\ge 5$, we have by Lemma~\ref{lemma:pascal}, $\binom{n+2}{3} + \binom{n+1}{3} \ge \binom{n+3}{3} - 1$.
Thus $m \le \binom{n+1}{3}$, and hence by Lemma~\ref{lemma:nogaps},
there is a permutation $\pi\in \symS_n$, with $\ninvsum{\pi} = m$.
Thus, by Theorem~\ref{thm:1dotpi},
\[
\mathbf{1}\cdot\pi  = \binom{n+2}{3} + \ninvsum(\pi) = \binom{n+2}{3} + m = k.
\]

For the values of $k$ less than 35, we first consider in the following chart for each $n\le 5$, the maximum and minimum values $\mathbf{1}\cdot \pi$ can obtain, where $\pi\in \symS_n$.
\[
\begin{array}{l|c|c}
n & \binom{n+2}{3} & \binom{n+2}{3}+\binom{n+1}{3}\bigstrut\\
\hline
1 & 1 & 1\\
2 & 4 & 5\\
3 & 10 & 14\\
4 & 20 & 30\\
5 & 35 & 55\\
\end{array}
\]
By Lemma~\ref{lemma:nogaps}, we have permutations $\pi$ such that the value $\mathbf{1}\cdot \pi$ can hit every value from 20 through 30.
For the other values, we have the following chart
\[
\begin{array}{c|c}
\pi & \mathbf{1}\cdot \pi\\
\hline
1 & 1\\
21 & 4\\
12 & 5\\
321 & 10\\
312 & 11\\
132 & 13\\
123 & 14
\end{array} \qedhere
\]
\end{proof}
Note that an integer $k\not\in\{2,3,6,7,8,9,12,15,16,17,18,19,31,32,33,34\}$ is even if and only if there is a permutation $\pi$ such that $\mathbf{1}\cdot \pi =k$ and the number of odd integers in the odd positions of $\pi$ is even.

\subsection{Algorithm}\label{section:algorithm}
We present an algorithm for finding a permutation $\pi$ for a given $k\not\in\{2,3,6,7,8,9,12,15,16,17,18,19,31,32,33,34\}$, such that $\mathbf{1}\cdot \pi = k$.
We first introduce three functions: $\eta$, $r$, and $\nu$.

For $k < 35$ and $k\not\in\{2,3,6,7,8,9,12,15,16,17,18,19,31,32,33,34\}$, let $\eta(k)$ be $\pi$ such that $1\cdot \pi = k$ (this is guaranteed by Lemma~\ref{thm:existPIforCOS} and is easy to make explicit because of the bound on $k$).

Let $k$ be such that we wish to find $\pi$ with $\mathbf{1}\cdot \pi = k$.
In the proof of Theorem~\ref{thm:existPIforCOS}, we chose the length $n$ of the to-be-constructed $\pi$, such that $\binom{n+2}{3} \le k$.
Since $6(\binom{n+2}{3} - k) = n^3+3n^2+2n-6k$, we can determine from $k$ the desired $n$ 
as the floor of the real cubic root of $n^3+3n^2+2n-6k$, which is the floor of
\begin{equation} \label{eq:3rdroot}
 \frac{1}{3}\sqrt[3]{81k + 3\sqrt{(27k)^2-3}}+\frac{1}{3}\sqrt[3]{81k - 3\sqrt{(27k)^2-3}} -1.
\end{equation}
Let $r$ be a function mapping a positive integer $k$ to such a value $n$.

Let $\nu:\{0,1,2,3,4,5,6,7,8,9,10\}\to \symS_4$, be given by $0\mapsto 4321$, $1\mapsto 3421$, $2\mapsto 3412$, $3\mapsto 4213$, $4\mapsto 4123$, $5\mapsto 2413$, $6\mapsto 3214$, $7\mapsto 1423$, $8\mapsto 2143$, $9\mapsto 1243$, $10\mapsto 1234$.  
This is from the proof of Lemma~\ref{lemma:nogaps4}.

Assuming the functions $\eta$, $r$, and $\nu$, we present an algorithm $\mathsf{Main}(k)$, see Algorithm~\ref{algo:Main},
that calls another function $\zeta$, defined below in Algorithm~\ref{algo:zeta}, that inputs $m$, a value for the ninvsum, and $n$, the length of the permutation to create.
\begin{algorithm}[h]
\caption{$\mathsf{Main}(k)$}
  \label{algo:Main}
\begin{algorithmic}
\IF{$k<35$}
\STATE output $\eta(k)$
\ELSE 
\STATE $n \gets r(k)$.
\STATE  $m \gets k - \binom{n+2}{3}$ (Note that $m \le \binom{n+1}{3}$.)
\STATE output $\zeta(m,n)$
\ENDIF
\end{algorithmic}
\end{algorithm}

\begin{algorithm}[h]
\caption{$\zeta(m,n)$}
  \label{algo:zeta}
\begin{algorithmic}
\IF{$n = 4$}
\STATE output $\nu(m)$ 
\ELSE\IF{$m\leq \binom{n}{3}$}
\STATE output $\zeta(m,n-1)\ominus 1$
\ELSE 
\STATE output $1\oplus \zeta(m-\binom{n}{2},n-1)$
\ENDIF\ENDIF
\end{algorithmic}
\end{algorithm}

Here $\pi\oplus\sigma$ is the \emph{direct sum} of the permutations $\pi$ and $\sigma$ and $\pi\ominus\sigma$ is the \emph{skew sum}.
Because of Lemma~\ref{lemma:pascal} and the fact that $n \ge 5$ for the first function call, we have that $m\le \binom{n+1}{3}$ for that first call.
The reasoning behind why the inductive hypothesis applies to Lemma~\ref{lemma:nogaps} guarantees that $m \le \binom{n+1}{3}$ for every function call after the first, even if $n = 4$. Also, because of equation~\ref{eq:3rdroot}, it is clear that the running time of this algorithm is proportional to $k^{1/3}$.

\section{Zone-crossing vectors and the distribution of the non-inversion sum} \label{sec:zcvs-and-ninvsumdist}
We are interested in the function
\[
N_n(x) = \sum_{\pi\in \symS_n} x^{\ninvsum(\pi)}
\]
which records the distribution of the non-inversion sum.
Table~\ref{tabl:N} provides some empirical data generated with the computer algebra system
Sage\footnote{www.sagemath.org}, where we factor the polynomials as much as possible. 
In the context of Table~\ref{tabl:N}, some of these polynomials factor into some reasonably small factors and a very large factor.
\begin{table}[h] \label{tabl:N}
\caption{The distribution function of $\ninvsum$, $N_n(x)$.}
\begin{center}\tiny
\begin{tabular}{|c|l|l|}
\hline
$n$ & Small factors of $N_n(x)$ & Big factor of $N_n(x)$ \bigstrut\\
\hline
$1$ & $1$ & $1$ \bigstrut\\
\hline
$2$ & $x+1$ & $1$ \bigstrut\\
\hline
$3$ & $1$ & $x^4 + 2x^3 + 2x + 1$ \bigstrut\\
\hline
$4$ & $x^2 + 1$ & $x^8 + 3x^7 + x^5 + 2x^4 + x^3 + 3x + 1$ \bigstrut\\
\hline
\multirow{2}{*}{$5$} & \multirow{2}{*}{$x^2 - x + 1$} & $x^{18} + 5x^{17} + 7x^{16} + 8x^{15} + 8x^{14} + 6x^{13} +
2x^{12} + 6x^{11} + 10x^{10}$ \bigstrut\\ 
& & $+ 14x^9 + 10x^8 + 6x^7 + 2x^6 + 6x^5 +
8x^4 + 8x^3 + 7x^2 + 5x + 1$ \\
\hline
\multirow{4}{*}{$6$} & \multirow{4}{*}{$(x + 1)(x^2 - x + 1)^2$} & $x^{30} + 6x^{29} + 11x^{28} + 13x^{27} + 13x^{26} +
6x^{25} - x^{24} + 6x^{23} + 21x^{22}$ \bigstrut\\
& & $ + 30x^{21}+ 19x^{20} + 3x^{19} - 7x^{18} +
14x^{17} + 27x^{16} + 36x^{15} + 27x^{14}$\\
& & $14x^{13} - 7x^{12} + 3x^{11} +
19x^{10} + 30x^9 + 21x^8 + 6x^7 - x^6 + 6x^5$\\
& & $ + 13x^4 + 13x^3 +
11x^2 + 6x + 1$ \\
\hline
\multirow{7}{*}{$7$} & \multirow{7}{*}{$(x^2 - x + 1)$} & $x^{54} + 7x^{53} + 16x^{52} + 23x^{51} + 36x^{50} + 39x^{49} +
38x^{48} + 45x^{47} + 62x^{46}$ \bigstrut\\
& & $ + 71x^{45} + 83x^{44} + 82x^{43} + 83x^{42} +
91x^{41} + 86x^{40} + 85x^{39} + 128x^{38}$\\
& & $ + 149x^{37} + 144x^{36} + 129x^{35}
+ 132x^{34} + 101x^{33} + 137x^{32} + 166x^{31} $\\
& & $+ 204x^{30} + 182x^{29} +146x^{28} + 108x^{27} + 146x^{26} + 182x^{25} + 204x^{24} $\\
& & $+ 166x^{23} +
137x^{22} + 101x^{21} + 132x^{20} + 129x^{19} + 144x^{18}+ 149x^{17}$\\
& & $ + 128x^{16} + 85x^{15} + 86x^{14} + 91x^{13} + 83x^{12} + 82x^{11} + 83x^{10} + 71x^9$\\
& & $ + 62x^8 + 45x^7 + 38x^6 + 39x^5 + 36x^4 + 23x^3 + 16x^2 +
7x + 1$ \\
\hline
\multirow{10}{*}{$8$} & \multirow{10}{*}{$(x^4 + 1)(x^2 - x + 1)$} & $x^{78} + 8x^{77} + 22x^{76} + 36x^{75} + 60x^{74} +
71x^{73} + 66x^{72} + 67x^{71} + 84x^{70}$ \bigstrut\\
& & $+ 94x^{69} + 133x^{68} + 150x^{67} +
171x^{66} + 182x^{65} + 164x^{64} + 135x^{63}$\\
& & $+ 196x^{62} + 249x^{61} +
280x^{60} + 278x^{59} + 290x^{58} + 218x^{57} + 243x^{56} $\\
& & $+ 270x^{55} + 375x^{54} + 456x^{53} + 432x^{52} + 326x^{51} + 322x^{50} + 329x^{49}$\\
& & $ + 442x^{48} + 481x^{47} + 533x^{46} + 464x^{45} + 413x^{44} + 362x^{43} +
437x^{42}$\\
& & $+ 489x^{41} + 520x^{40} + 462x^{39} + 520x^{38} + 489x^{37} +
437x^{36} + 362x^{35}$\\
& & $ + 413x^{34} + 464x^{33} + 533x^{32} + 481x^{31} +
442x^{30} + 329x^{29} + 322x^{28}$\\
& & $ + 326x^{27} + 432x^{26} + 456x^{25} +
375x^{24} + 270x^{23} + 243x^{22} + 218x^{21}$\\
& & $ + 290x^{20} + 278x^{19} +
280x^{18} + 249x^{17} + 196x^{16} + 135x^{15} + 164x^{14}$\\
& & $+ 182x^{13} +171x^{12} + 150x^{11} + 133x^{10} + 94x^9 + 84x^8 + 67x^7 + 66x^6$\\
& & $ +71x^5 + 60x^4 + 36x^3 + 22x^2 + 8x + 1$ \\
\hline
\end{tabular}
\end{center}
\label{default}
\end{table}%

One can observe that the degree of $N_n(x)$ is always the $(n-1)\uth$ tetrahedral number $\binom{n+1}{3}$.
This is consistent with equation~(\ref{eq:invs-ninvs}), where the maximum non-inversion sum $\binom{n+1}{3}$ can be obtained using the identity permutation $\mathbf{1}$.

The primary aim of this section is to find a recursive definition of the distribution function for the non-inversion sum, that is, to define $N_{n+1}(x)$ in terms of $N_n(x)$.
Our formulation of the distribution function will involve a new type of vector, the zone-crossing vector, whose coordinates count the number of inversions or non-inversions $(a,b)$ of a permutation, with a given point between $a$ and $b$.
\begin{definition}
Given a permutation $\pi$ of rank $n$, we define
\begin{enumerate}
\item  its \emph{inversion zone-crossing vector}, $\izcv(\pi) = (z_1,z_2,\ldots,z_{n-1})$, where $z_k$ is the number of inversions $(a,b)\in \INV(\pi)$, where $a\le k<b$, and its \emph{augmented zone crossing vector} $\aizcv(\pi) = (0,z_1,z_2,\ldots,z_{n-1},0)$.
\item its \emph{non-inversion zone-crossing vector} $\nzcv(\pi) = (z_1,z_2,\ldots,z_{n-1})$, where $z_k$ is the number of non-inversions $(a,b)\in \NINV(\pi)$, where $a\le k<b$, and its \emph{augmented zone crossing vector} $\anzcv(\pi) = (0,z_1,z_2,\ldots,z_{n-1},0)$.
\end{enumerate}
\end{definition}

\begin{example}
Consider the permutation $\pi = 314562$. Then $\izcv(\pi) = (2,1,2,3,4)$ and
$\zcv(\pi) = (3,7,7,5,1)$.

\end{example}

The following proposition states that a zone crossing vector uniquely determines its permutation.
\begin{proposition}
If $v  = (v_0,v_1,\dotsc,v_{n-1},v_n) = \aizcv(\pi)$, then $\pi_k = n-(k-1)-(v_k-v_{k-1})$, for $1\le k\le n$.
(Therefore, if $\rho$ is a permutation, such that $\aizcv(\pi) = \aizcv(\rho)$, then $\pi = \rho$.)
\end{proposition}
\begin{proof}
Let $\rho$ be a permutation, and let $v$ be its zone-crossing vector.
Let $\pi$ be constructed according to the statement of the proposition.
We show that $\pi = \rho$.
First observe that $\rho_1 = n-v_1$, since this is the number of positions to the right of the first position that have a value greater than $\rho_1$.
Thus $\rho_1 = \pi_1$.
For a general $k\ge 1$, note that if $\rho_k = 1$, then $v_k-v_{k-1} = n-k$.
Thus $1 = \rho_k = n - (k-1) - (v_k-v_{k-1})$, just as is the case with $\pi_k$.
To consider different values of $\rho_k$, imagine incrementing its value by 1 as a result of swapping $\rho_k$ with the position with one larger value.
If $\rho_k$ is incremented by 1, then $v_k - v_{k-1}$ is decremented by 1, for either the original value of $\rho_k$ is swapped with a value to the right, thus decrementing $v_k$, or it is swapped with a value to the left, thus incrementing $v_{k-1}$.
Thus all the values of $\rho_k$ can be obtained by the formula above, and hence $\rho= \pi$.
\end{proof}

\begin{lemma}
The sum of the coordinates of $\zcv(\pi)$ equals $\ninvsum(\pi)$.
\end{lemma}
\begin{proof}
This follows from the fact that each non-inversion will contribute to as many zone-crossing coordinates as is the separation distance of the non-inversion.  For example, a non-inversion from position 1 to position 3 has separation 2, which is the number of zone-crossing coordinates it will contribute to.
\end{proof}

\begin{proposition}
For any $\pi\in \symS_n$,
\begin{enumerate}
\item $\zcv(\pi)+\zicv(\pi) = (1\cdot n-1, 2\cdot n-2,\dotsc, n-1\cdot 1)$, 
\item $\zcv(\pi^\rmc) = \zicv(\pi)$,
\item $\zcv(\pi^\rmr) = \zicv^\rmr(\pi)$.
\end{enumerate}
\end{proposition}
\begin{proof}
\begin{enumerate}
\item To prove $\zcv(\pi)+\zicv(\pi) =(1\cdot n-1, 2\cdot n-2,\dotsc, n-1\cdot 1)$, we note that the $j\uth$ coordinate of $\zcv(\pi) +\zicv(\pi)$ counts the total number of pairs matching each coordinate in the first $j$ positions with each coordinate in the last $n-j$ positions. 
This is because each such pair is a zone crossing inversion or non-inversion and is hence counted in either $\zicv(\pi)$ or $\zcv(\pi)$. 
This yields the vector $(1\cdot n-1, 2\cdot n-2,\dotsc, n-1\cdot 1)$, giving us the desired formula.
\item To prove $\zcv(\pi^\rmc) = \zicv(\pi)$, note that the complement operation changes every inversion to a non-inversion, and every non-inversion to an inversion.
\item To prove $\zcv(\pi^\rmr) = \zicv^\rmr(\pi)$, note that every inversion in $\pi$ between positions $j$ and $j+k$ is a non-inversion in $\pi^\rmr$ from positions $n-j-k+1$ to $n-j+1$.  This yields $\zcv^\rmr(\pi^\rmr) = \zicv$, and reversing the vectors on each side yields the desired equation.\qedhere
\end{enumerate}
\end{proof}

\begin{lemma}\label{lemma:zonecrossingrecursion}
Let $\pi$ be a permutation of rank $n$ with augmented zone-crossing vector $\anzcv(\pi) = (a_0,a_1, a_2, \dotsc, a_{n-1},a_n)$.
Let $\rho$ be the permutation obtained by inserting $n+1$ into $\pi$ in between position $k$ and $k+1$
(in the case where $k = 0$, the resulting permutation is $1\ominus \pi$).
Then for $0 \le k \le n$,
\[
\anzcv(\rho) = (a_0+0,a_1 + 1, a_2 + 2, \dotsc, (a_k + k), a_k, a_{k+1}, \dotsc, a_n).
\]
(Note that for $k = 0$, we have $\anzcv(\rho) = (0,0, a_1, a_2, \dotsc, a_{n-1},0)$ and for $k = n$ we have
$\anzcv(\rho) = (0,a_1+1, a_2+2, \dotsc, a_{n-1} + n-1,n,0)$.)
\end{lemma}
\begin{proof}
For $0 \leq j \leq k$ the $j\uth$ position of the zone-crossing vector (counting from 0 in the augmented vector) is incremented by the number $j$ of positions in the left zone, as each forms a new non-inversion pair with the new position $k+1$ in the right zone.
The $j\uth$ position among the $(k+1)\uth$ position in the new zone-crossing vector counts the number of non-inversions starting among the first $j$ positions and ending among the $n+1-j$ positions.  As the inserted position is now in the left zone and has the highest value, it does not contribute to the zone-crossing count.
Thus the $j\uth$ position of the new zone-crossing vector is the same as the $(j-1)\uth$ position of the old zone-crossing vector (we decrement the position by one, as the position counts the size of the left zone, which decreases by one when the inserted position is removed).
\end{proof}

The preceding lemma can be used to give a recursive definition of the zone-crossing vectors. 
We show how this is done when adding a value to permutations of rank $2$ to obtain permutations of rank $3$.
The set of (augmented) non-inversion zone-crossing vectors for permutations of rank $2$ are $(010)$ corresponding to the permutation $12$ and $(000)$ corresponding to $21$.
Using Lemma \ref{lemma:zonecrossingrecursion}, we determine the (augmented) non-inversion zone crossing vectors for permutations of rank $3$ in the following chart:  
$$
\begin{array}{l|c|c|c}
& k= 0& k=1& k=2\\
\hline
(000) & (0\underline{0}00) & (01\underline{0}0) & (012\underline{0}) \\
(010) & (0\underline{0}10) & (02\underline{1}0) & (022\underline{0})
\end{array}
$$
The first column gives the zone-crossing vectors for the permutations of rank $2$.
The first row gives the position $k$ to the right of which we place the highest value to obtain the new permutation.
The remaining entries correspond to the resulting zone-crossing vectors.
We underline the value $a_k$ in position $k+1$ of the zone-crossing vector.
Notice that these correspond to the positions $k$ in the left column.

In summary, we see that these vectors do not range across all possibilities between the lowest $(0000)$ and the highest $(0220)$, as we are missing $(0110)$, $(0020)$ and $(0200)$.
We hope future work can yield a more direct characterization of the set of all possible zone-crossing vectors.
Such a characterization may reveal new patterns of a variety of permutation statistics.
A variation of such a characterization might, for example, capture sets of zone-crossing vectors correspond to permutations avoiding a certain classical pattern of rank 3, such as the patten $231$.
Since the number of permutations of rank $n$ that avoid a given classical pattern of rank 3 is the $n\uth$ Catalan number, such a set may offer a new Catalan structure.

The proof of the following theorem will make use of partitions of integers. Given a positive integer, $n$, we write $\lambda \dashv n$ to indicate that $\lambda$ is a partition of $n$.
We write $\ell_\lambda$ for the length of the partition $\lambda$. For example, $\lambda = 3+3+2+1$ is
a partition of $9$ of length $4$.

\begin{theorem} \label{thm:zcvdist}
The number of permutations of rank $n$ such that the $k\uth$ coordinate of
the zone-crossing vector equals $\ell$ is
\[
k!(n-k)![q^\ell]\qchoose{n}{k}.
\]
In other words
\[
\sum_{\pi \in \symS_n}q^{\zcv(\pi)_k} = k!(n-k)! \qchoose{n}{k}.
\]
\end{theorem}
\begin{proof}
The number of partitions of $\ell$ into at most $k$ parts, where each part has size at most $n-k$, is given by $[q^\ell]\qchoose{n}{k}$.  
Each part may correspond to one of the first $k$ positions of a permutation, and the size of the part would correspond to the number out of the $n-k$ last positions that the selected position is a non-inversion with.
For each partition, we may rearrange the first $k$ positions and rearrange the last $n-k$ positions without affecting the $k\uth$ coordinate of the zone-crossing vector.
This gives us the remaining $k!(n-k)!$.
\end{proof}

\begin{theorem} \label{thm:ninvsumdist}
For $n \geq 1$, letting $\binom{1}{2} = 0$, we have
\begin{align*}
N_{n+1}(q) &= \sum_{k=0}^{n} q^{\binom{k+1}{2}}\sum_{\pi\in \symS_n}q^{\anzcv_k(\pi)}q^{\ninvsum(\pi)}\\
&= N_n(q) +  \sum_{k=1}^{n-1} q^{\binom{k+1}{2}}\sum_{\pi\in \symS_n}q^{\zcv_k(\pi)}q^{\ninvsum(\pi)} + q^{\binom{n+1}{2}}N_n(q).
\end{align*}
\end{theorem}

\begin{proof}
We are interested in the result of extending a permutation $\pi$ to a permutation $\rho$ by inserting $n+1$ between the $k\uth$ and $(k+1)\uth$ positions of $\pi$ (like in Lemma~\ref{lemma:zonecrossingrecursion}).
We proceed by summing the functions restricted to permutations with $n+1$ in the $(k+1)\uth$ positions for each $k$.
When $k= 0$, the value $n+1$ is inserted at the beginning of the permutation and hence contributes nothing to the non-inversion sum.
Thus we have the term $N_n(q)$.
When $k = n$, the value $n+1$ is in the last position, and hence adds (the maximum) $\binom{n+1}{2}$ to whatever non-inversion sum the permutation originally had.
Thus our last term will be $q^{\binom{n+1}{2}}N_n(q)$.
For the other values of $k$ ($1\le k\le n-1$), what the insertion of the value $n+1$ contributes to the non-inversion sum depends on the original permutation.
By Lemma~\ref{lemma:zonecrossingrecursion}, the sum of the zone-crossing vector coordinates (equaling the non-inversion sum) increases by $\binom{k+1}{2}+\nzcv_k$, which is why we multiply $q^{\ninvsum(\pi)}$ by $q^{\binom{k+1}{2}+\nzcv_k(\pi)}$.
\end{proof}

\section{The distribution of $k$-, $(k_1,k_2)$-, and $(\leq k)$-step inversions} \label{subsec:Hnk}

\subsection{The distribution of $k$-step inversions}

\begin{definition}
A \emph{$k$-step inversion} of a permutation $\pi$ is an inversion $(a,b)\in \INV(\pi)$, where $b-a=k$.
Similarly a \emph{$k$-step non-inversion} is a non-inversion $(a,b)\in \NINV(\pi)$, such that $b-a = k$.
\end{definition}

Let $\inv_k(\pi)$ be the number of $k$-step inversions in $\pi$, and let $\ninv_k(\pi)$ be
the number of $k$-step non-inversions in $\pi$. Then
$\inv(\pi) = \sum_{k=1}^{n-1} \inv_k(\pi)$, and similarly for $\ninv_k(\pi)$. Define
\[
H_{n,k}(x) = \sum_{\pi \in \symS_n} x^{\inv_k(\pi)}
\]
so $I(n,k,i) \eqdef [x^i]H_{n,k}(x)$ is the number of permutations in $\symS_n$ with the number of $k$-step inversions equaling the number $i$.
It is known that $H_{n,1}(x)$ is the $n\uth$ Eulerian polynomial, which we denote $A_n(x)$\footnote{The coefficient $[x^{k-1}]A_n(x)$ is $T(n,k)$ in the sequence http://oeis.org/A008292}, since a $1$-step inversion is a descent.

In finding $H_{n,k}(x)$ for arbitrary $k$, we will divide up the permutations into $k$ smaller permutations which can be interleaved to form the original.
We call these smaller permutations \emph{runs}, and define them precisely as follows.
\begin{definition}
Given a permutation $\pi$ of rank $n$ and $1\le k \le n$, the $i\uth$ $k$-step run of $\pi$ is the permutation $\rho$ of rank $j=\lfloor (n-i)/k \rfloor+1$, where $\rho_j = \pi_{k(j-1)+i}$.
\end{definition}
Let $\lambda_i =  \lfloor (n-i)/k \rfloor+1$ be the length of the $i\uth$ $k$-step permutation.
One can observe that if $n\ge k$, then there are $\rem(n/k)$ many $j$, such that $\lambda_j = \lfloor n/k\rfloor +1$, and $k - \rem(n/k)$ many $j$, such that $\lambda_j = \lfloor n/k \rfloor$. 
Furthermore, $\sum_{i=1}^k \lambda_i = n$. 
The intuition for this can be seen in the following example, where the runs partition the original permutation, and hence the $\lambda_i$ partition $n$.
\begin{example}
Consider the case $n = 11$, $k = 4$.
Since $n\ge k$, the total number of $4$-step runs is four.  
Of those, three are of length $3$ ($\lambda_1,\lambda_2,\lambda_3 = 3$).
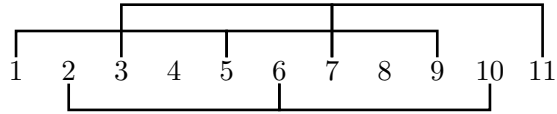
\[
\pairsWsWl{}{1,...,11}{1,5}{3,7}{2,6}{}{2}{4}
\]
The remaining one is of length $2$ ($\lambda_4 = 2$).
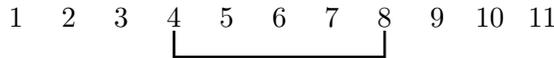
\[
\pairsWsWl{}{1,...,11}{}{}{4}{}{2}{4}
\]
Note that $k$-step inversions only occur within the same $k$-step run, and that a $k$-step inversion in the original permutation corresponds to a $1$-step inversion (a descent) in a run.
We will see in the proof of the next theorem how this leads to $H_{11,4} = I(11,4,0) A_3^3(x) A_2^1(x),$ where $I(11,4,0)$ will count the ways of distributing the $11$ values among the $4$ runs, and the $A_s$ will correspond to the permutations of the $s$ values within a run.  
\end{example}
\begin{theorem} \label{thm:H_nk}
For $1 \leq k \leq n$ let $s = \lfloor n/k \rfloor +1$ and $t = \rem(n/k)$.
Then
\[
H_{n,k}(x) = I(n,k,0)A^t_{s}(x)A^{k-t}_{s-1}(x),
\]
where $A_\ell(x) = H_{\ell,1}(x)$ is the $\ell\uth$ Eulerian polynomial, 
and
\[
I(n,k,0) =
\prod_{j=1}^{k-1} \binom{n - \sum_{i=0}^{j-1}\lambda_i}{\lambda_j},
\]
where the $\lambda_0 \eqdef 0$ and for $1\le j\le k-1$, $\lambda_j$ is the length of the $j\uth$ $k$-step run.
\end{theorem}

\begin{proof}
Let us first find $I(n,k,0)$.
Given a permutation $\pi$ with no $k$-step inversions, the order within each run must be increasing. 
Thus the entire variation of such permutations is with how the $n$ values are distributed among the runs.
If $j-1$ runs have been filled, we must choose $\lambda_j$ more out of the remaining $n - \sum_{i=0}^{j-1}\lambda_i$.
Thus we have
\[
I(n,k,0) = \prod_{j=1}^{k-1} \binom{n - \sum_{i=0}^{j-1}\lambda_i}{\lambda_j}.
\]
Furthermore, the number of $k$-step inversions is invariant over how we distribute $n$ among the runs.
Thus $I(n,k,0)$ is a factor of the distribution function.
What the number of $k$-step inversions depends on is how the numbers are arranged within each run.
Note again that $k$-step inversions only occur within the same $k$-step run, and that a $k$-step inversion in the original permutation corresponds to a $1$-step inversion (a descent) in a run.
In this way, the runs do not interact.
Thus 
\[
I(n,k,i) = \sum_{\sum_{j=1}^k p_j = i;\; 1\le p_j<\lambda_j} I(n,k,i)\prod_{j=1}^k[x^{p_j}]A_{\lambda_j}(x).
\]
Since there are $t=\rem(n/k)$ many $j$, such that $\lambda_j = s = \lfloor n/k\rfloor +1$, and $k - \rem(n/k)$ many $j$, such that $\lambda_j = \lfloor n/k \rfloor$, we have that 
\[
H_{n,k}(x) = I(n,k,0)A^t_{s}(x)A^{k-t}_{s-1}(x). \qedhere
\] 
\end{proof}

If we had used $\ninv_k$ instead of $\inv_k$ in the definition of the distribution function
$H_{n,k}(x)$ then we would have obtained the same formula as in the theorem above. Also, had
we used the difference of values, rather than positions, in the definitions of $\inv_k$ and
$\ninv_k$, we would also have arrived at the same formula, since the values-definition for $\pi$ would
have corresponded to the positions-definition for $\pi^\rmi$.
Table~\ref{table:Kstep} includes experimental runs for the distribution function $H_{n,k}$ for $n = 1,\dotsc,9$, and select $k$ for high values of $n$.

\begin{table}[h]
\caption{The distribution function of $\ninv_k$, $H_{n,k}(x)$.}
\begin{center}\tiny
\begin{tabular}{|c|c|l|}
\hline
$n$ & $k$ & $H_{n,k}(x)$ \bigstrut\\
\hline
$1$ & $1$ & $1$ \bigstrut\\
\hline
\multirow{2}{*}{$2$} & $1$ & $x+1$ \bigstrut\\
    & $2$ & $2$ \\
\hline
\multirow{3}{*}{$3$} & $1$ & $x^2+4x+1$ \bigstrut\\
    & $2$ & $3(x+1)$ \\
    & $3$ & $6$ \\
\hline
\multirow{4}{*}{$4$} & $1$ & $(x+1)(x^2+10x+1)$ \bigstrut\\
    & $2$ & $6(x+1)^2$ \\
    & $3$ & $12(x+1)$ \\
    & $4$ & $24$ \\
\hline
\multirow{5}{*}{$5$} & $1$ & $x^4+26x^3+66x^2+26x+1$ \bigstrut\\
    & $2$ & $10(x+1)(x^2+4x+1)$ \\
    & $3$ & $30(x+1)^2$ \\
    & $4$ & $60(x+1)$ \\
    & $5$ & $120$ \\
\hline
\multirow{6}{*}{$6$} & $1$ & $(x+1)(x^4+56x^3+246x^2+56x+1)$ \bigstrut\\
    & $2$ & $20(x^2+4x+1)^2$ \\
    & $3$ & $90(x+1)^3$ \\
    & $4$ & $180(x+1)^2$ \\
    & $5$ & $360(x+1)$ \\
    & $6$ & $720$ \\
\hline
\multirow{7}{*}{$7$} & $1$ & $x^6+120x^5+1191x^4+2416x^3+1191x^2+120x+1$ \bigstrut\\
    & $2$ & $35(x+1)(x^2+4x+1)(x^2+10x+1)$ \\
    & $3$ & $210(x+1)^2(x^2+4x+1)$ \\
    & $4$ & $630(x+1)^3$ \\
    & $5$ & $1260(x+1)^2$ \\
    & $6$ & $2520(x+1)$ \\
    & $7$ & $5040$ \\
\hline
\multirow{8}{*}{$8$} & $1$ & $(x+1)(x^6+246x^5+4047x^4+11572x^3+4047x^2+246x+1)$ \bigstrut\\
    & $2$ & $70(x+1)^2(x^2+10x+1)^2$ \\
    & $3$ & $560(x+1)(x^2+4x+1)^2$ \\
    & $4$ & $2520(x+1)^4$ \\
    & $5$ & $5040(x+1)^3$ \\
    & $6$ & $10080(x+1)^2$ \\
    & $7$ & $20160(x+1)$ \\
    & $8$ & $40320$ \\
\hline
\multirow{9}{*}{$9$} & $1$ & $x^8+502x^7+14608x^6+88234x^5+156190x^4+88234x^3+14608x^2+502x+1$ \bigstrut\\
    & $2$ & $126(x+1)(x^2+10x+1)(x^4+26x^3+66x^2+26x+1)$ \\
    & $3$ & $1680(x^2+4x+1)^3$ \\
    & $4$ & $7560(x+1)^3(x^2+4x+1)$ \\
    & $\vdots$ & $\vdots$  \\
\hline
\end{tabular}
\end{center}
\label{table:Kstep}
\end{table}%

\subsection{$(k_1,k_2)$-step inversions and non-inversions}
\begin{definition}
Given a permutation $\pi$, a \emph{$(k_1,k_2)$-step inversion} is a pair $(a, b)$, such that
$1 \leq a < b \leq n$, satisfying $b-a = k_1$ and $\pi(b)-\pi(a) = k_2$.
Similarly, a \emph{$(k_1,k_2)$-step non-inversion} is a pair $(a, b)$, such that
$1 \leq a < b \leq n$, satisfying $b-a = k_1$ and $\pi(a)-\pi(b) = k_2$.
\end{definition}

Let $\inv_{(k_1,k_2)}(\pi)$ be the number of $(k_1,k_2)$-step inversions in $\pi$. Then
\[
\inv(\pi) = \sum_{1\leq k_1,k_2 \leq n-1} \inv_{(k_1,k_2)}(\pi).
\]
Define
\[
H_{n,(k_1,k_2)}(x) = \sum_{\pi \in \symS_n} x^{\inv_{(k_1,k_2)}(\pi)}.
\]

\begin{proposition}\label{proposition:k1k2specialcase}
Let $n/2 < k_1,k_2 < n$.
Then the degree of $H_{n,(k_1,k_2)}(x)$ is $\ell = \min(n-k_1,n-k_2)$, and its leading coefficient equals
\[
(n-2\ell)!\ell! \binom{n-k_1}{\ell}\binom{n-k_2}{\ell}.
\]
\end{proposition}
\begin{proof}
Since $k_1>n/2$, there are $n-k_1$ location pairs that a $k_1$-step inversion could be, since there are $n-k_1$ many $k$-step runs of length two, with the remainder of the runs of length $1$.
Similarly, since $k_2>n/2$, there are at most $n-k_2$-inversions with a value separation of $k_2$, as the top value has to be greater than $k_2$. 
Thus $\ell = \min(n-k_1,n-k_2)$ is the maximum number of $(k_1,k_2)$-step inversions, and the degree of $H_{n,(k_1,k_2)}(x)$ is $\ell$.

For the leading coefficient, we select $\binom{n-k_1}{\ell}$ positions pairs to place top values among $\binom{n-k_2}{\ell}$.
Then there are $\ell!$ ways to arrange the values among the position pairs, and there are $(n-2\ell)!$ ways to arrange the remaining values among the remaining positions.
Note that no new $(k_1,k_2)$-step inversions can occur with the $(n-2\ell)!$ values and positions, since either all the available position pairs have been filled (when $\ell= n-k_1$) or all the available top values have been used (when $\ell = n-k_2)$.
\end{proof}
This proof will be adapted in Section \ref{section:topsModd}, for a similar result involving $k$-step inversions with inversion tops divisible by $d$.
For future work, we would like a complete description of the polynomials $H_{n,(k_1,k_2)}(x)$ as we had for $H_{n,k}(x)$, the distribution of $k$-step inversions.
Table~\ref{table:k1k2stepruns} contains some experimental runs for $H_{n,(k_1,k_2)}(x)$. 

\begin{table}[h]
\caption{The distribution function of $\inv_{(k_1,k_2)}$, $H_{n,(k_1,k_2)}(x)$.}
\begin{center}\tiny
\begin{tabular}{|c|c|l|l|l|l|}
\hline
$n$ & $k_1$ & $k_2 = 1$ & $k_2 = 2$ & $k_2 = 3$ & $k_2 = 4$ \\
\hline
$1$ & $1$ & $1$ & & & \\
\hline
$2$ & $1$ & $x+1$ & 2 & & \\
    & $2$ & $2$   & 2 & & \\
\hline
$3$ & $1$ & $x^2+2x+3$ & $2(x+2)$ & $6$ & \\
    & $2$ & $2(x+2)$   & $x+5$    & $6$ & \\
    & $3$ & $6$        & $6$      & $6$ & \\
\hline
$4$ & $1$ & $x^3+3x^2+9x+11$ & $2(x^2+4x+7)$ & $6(x+3)$  & $24$ \\
    & $2$ & $2(x^2+4x+7)$    & $2(x^2+2x+9)$ & $4(x+5)$  & $24$ \\
    & $3$ & $6(x+3)$         & $4(x+5)$      & $2(x+11)$ & $24$ \\
    & $4$ & $24$             & $24$          & $24$      & $24$ \\
\hline
$5$ & $1$ & $x^4+4x^3+18x^2+44x+53$ & $2(x^3+6x^2+21x+32)$ & $6(x^2+6x+13)$  & $24(x+4)$ \\
    & $2$ & $2(x^3+6x^2+21x+32)$    & $(x+3)(x^2+4x+25)$   & $4(x^2+7x+22)$  & $6(3x+17)$ \\
    & $3$ & $6(x^2+6x+13)$          & $4(x^2+7x+22)$       & $2(x^2+10x+49)$ & $12(x+9)$ \\
    & $4$ & $24(x+4)$               & $6(3x+17)$           & $12(x+9)$       & $6(x+19)$ \\
    & $5$ & $120$                   & $120$                & $120$           & $120$ \\
\hline
\end{tabular}
\end{center}
\label{table:k1k2stepruns}
\end{table}%

\subsection{The distribution of $(\leq k)$-step inversions}
\begin{definition}
Let
\[
\inv_{\leq k}(\pi) = \sum_{k' \leq k} \inv_{k'}(\pi).
\] 
\end{definition}
Then,
since $n-1$ is the maximum separation, and a separation of one corresponds to a descent, we get
\[
\inv(\pi) = \inv_{\leq n-1}(\pi) \quad\text{and}\quad \des(\pi) = \inv_{\leq 1}(\pi),
\]
for any permutation of rank $n$.
Define
\[
J_{n,\leq k}(x) = \sum_{\pi \in \symS_n} x^{\inv_{\leq k}(\pi)}.
\]

For the purpose of the next proposition we recall the \emph{falling factorial}
\[
(k)_j = k (k-1) \dotsm (k-(j-1)),
\]
and define a differential operator
\[
\nabla_k = \sum_{j=0}^k \frac{j+1}{(k)_j} \frac{\rmd^j}{\rmd x^j}.
\]

\begin{proposition} \label{prop:leq_k-step:n-1case}
\begin{enumerate}
\item
If the maximum step-size is $1$, we have
\[
J_{n,1}(x) = A_n(x),
\]
where $A_n(x)$ is the $n\uth$ Eulerian polynomial.
\item
If the maximum step-size is $n-2$, we have
\[
J_{n,\leq n-2}(x) = J_{n-1,\leq n-3}(x)
\cdot
\left(
\frac{\rmd}{\rmd x} \left( x \sum_{j=0}^{n-2} x^j \right)
+
x^{n-2} \nabla_{n-2} (x^{n-2})
\right).
\]

\item \label{thisprop:case3}
If the maximum step-size is $n-1$, we have
\[
J_{n,\leq n-1}(x) = J_{n-1,\leq n-2}(x)
\cdot \sum_{j=0}^{n-1} x^j = [n]_x !.
\]
\end{enumerate}
\end{proposition}
\begin{proof}
\begin{enumerate}
\item[(1)]
This follows from the fact that a $1$-step inversion is a descent.

\item[(3)]
As noted above $\inv_{\leq n-1}(\pi) = \inv(\pi)$ and therefore
\begin{align*}
J_{n,\leq n-1}(x) &= \sum_{\pi \in \symS_n} x^{\inv(\pi)} \\
&= (1+x)(1+x+x^2)\dotsm(1+x+x^2+\dotsb+x^{n-1})\\
& = [n]_x !.
\end{align*}
We also give an alternative proof:
Let $k = n-1$. For each value $m$ for the last position of a permutation $\sigma$, we have that $\inv_{\leq k}(\sigma) = n-m + \inv_{\leq k}(\tau)$, 
where $\tau$ is the permutation obtained by flattening the restriction of $\sigma$ the domain to $\{1,\ldots,n-1\}$.
Thus for each value $m$, we multiply $J_{n-1,\leq k-1}(x)$ by $x^{n-m}$ to account for all permutations that end in $m$.
We then add these products together for all values of $m$ so as to account for all permutations.

\item[(2)]
Let $k = n-2$. Here we imagine the effect of the first and the last positions on the inversion count of the middle positions.
For each pair $(m_f,m_\ell)$ of values that the first and last positions can assume, the contribution to the counts in $J_{n,\leq k}(x)$ will be the same as their contribution of the counts in $J_{n,\leq k+1}(x)$ as long as $m_f<m_\ell$.
Otherwise (if $m_f>m_\ell$), the contribution to $J_{n,\leq k}(x)$ is one less than it would be for $J_{n,\leq k+1}(x)$, since the first and last positions form an inversion not counted in the former, but counted in the latter.

By part~\eqref{thisprop:case3} of this proposition, the last two factors of $J_{n,\leq k+1}(x)$ are $(1+x^2+\cdots +x^k)$ and $(1+x^2+\cdots + x^k+k^{k+1})$, which when multiplied together give us
\begin{align}\label{eq:Jnk+1}
& 1+2x+\dotsb+(k-1)x^{k-2}+kx^{k-1} + (k+1)x^k\\
&+(k+1)x^{k+1}+\dotsb+3x^{2k-1}+2x^{2k}+x^{2k+1}\nonumber
\end{align}
The coefficient of each $x^j$ in (\ref{eq:Jnk+1}) corresponds to the number of pairs $(m_f,m_\ell)$ that contribute $j$ to the inversion count of the middle positions.
The number of inversions a pair $(m_f,m_\ell)$ contributes is $(m_f -1)$ from the first position plus $(n-m_\ell)$ from the last position minus possibly one for over counting in the case that $m_f>m_\ell$.
So the contribution of the pair to $J_{n,\leq k+1}(x)$ would be $n+ m_f -m_\ell -1$ if $m_f<m_\ell$ or $n+m_f-m_\ell - 2$ if $m_f>m_\ell$.
Note that if $m_f< m_\ell$, then this number is at most $n-2 = k$.
Otherwise (if $m_f> m_\ell$) this number is at least $n-1 = k+1$.
Thus, up to $j = k$, all pairs $(m_f,m_\ell)$ are such that $m_f<m_\ell$, and hence the exact same pairs can be used in the count for determining $J_{n,\leq k}$.
Starting with $j= k+1$, all pairs $(m_f,m_\ell)$, are such that $m_f>m_\ell$, and there is an inversion counted toward $J_{n,\leq k+1}(x)$ that is not counted toward $J_{n,k}(x)$, and hence these pairs will contribute to the case where $j=k$ when constructing the formula for $J_{n,\leq k}(x)$.
A similar argument shows that this generalizes, so that when $t\ge 1$, we have $x^{k+t}$ in (\ref{eq:Jnk+1}) replaced by $x^{k+t-1}$. 
Thus we obtain,  
\begin{align*}
J_{n,\leq k}(x) &= J_{n-1,\leq k-1}(x)
\cdot
(1+2x+\dotsb+(k-1)x^{k-2}+kx^{k-1} \\
&\hspace*{12ex}+ 2(k+1)x^k+kx^{k+1}+\dotsb+2x^{2k-1}+x^{2k}). \\
&= J_{n-1,\leq k-1}(x)
\cdot
\left( \sum_{j=0}^{k} (j+1) \left( x^{j} + x^{2k-j} \right) \right) \\
&= J_{n-1,\leq k-1}(x)
\cdot
\left(\sum_{j=0}^k (j+1) x^j + \sum_{j=0}^k (j+1)x^{2k-j}\right) \\
&= J_{n-1,\leq k-1}(x)
\cdot
\left(\sum_{j=0}^k (j+1) x^j + x^k\sum_{j=0}^k (j+1)x^{k-j}\right) \\
&=J_{n-1,\leq k-1}(x)
\cdot
\left(
\frac{\rmd}{\rmd x} \left( x \sum_{j=0}^{k} x^j \right)
+
x^k \nabla_k (x^k)
\right).\qedhere
\end{align*}
\end{enumerate}
\end{proof}

Table~\ref{table:leKstep} includes experimental runs for the distribution function $J_{n,\leq k}(x)$ for $n = 1, \dotsc, 7$.
\begin{table}[h]
\caption{The distribution function of $\inv_{\leq k}$, $J_{n,\leq k}(x)$.}
\begin{center}\tiny
\begin{tabular}{|c|c|l|}
\hline
$n$ & $k$ & $J_{n,\leq k}(x)$ \bigstrut\\
\hline
$1$ & $1$ & $1$ \bigstrut\\
\hline
$2$ & $1$ & $x+1$ \bigstrut\\
\hline
\multirow{2}{*}{$3$} & $1$ & $x^2+4x+1$
	\bigstrut\\\cline{2-3}
    & $2$ & $(x + 1)(x^2 + x + 1)$ \bigstrut\\
\hline
\multirow{3}{*}{$4$} & $1$ & $(x + 1)(x^2 + 10x + 1)$
	\bigstrut\\\cline{2-3}
    & $2$ & $(x + 1)(x^4 + 2x^3 + 6x^2 + 2x + 1)$
    \bigstrut\\\cline{2-3}
    & $3$ & $(x + 1)(x^2 + x + 1)(x^3 + x^2 + x + 1)$ \bigstrut\\
\hline
\multirow{4}{*}{$5$} & $1$ & $x^4 + 26x^3 + 66x^2 + 26x + 1$
	\bigstrut\\\cline{2-3}
    & $2$ & $(x + 1)^3(x^4 + x^3 + 11x^2 + x + 1)$
    \bigstrut\\\cline{2-3}
    & $3$ & $(x + 1)(x^2 + x + 1)(x^6 + 2x^5 + 3x^4 + 8x^3 + 3x^2 + 2x + 1)$
    \bigstrut\\\cline{2-3}
    & $4$ & $(x + 1)(x^2 + x + 1)(x^3 + x^2 + x + 1)(x^4 + x^3 + x^2 + x + 1)$ \bigstrut\\
\hline
\multirow{5}{*}{$6$} & $1$ & $(x + 1)(x^4 + 56x^3 + 246x^2 + 56x + 1)$
	\bigstrut\\\cline{2-3}
    & $2$ & $(x + 1)(x^8 + 4x^7 + 25x^6 + 88x^5 + 124x^4 + 88x^3 + 25x^2 + 4x + 1)$
    \bigstrut\\\cline{2-3}
    & $3$ & $(x + 1)^2(x^{10} + 3x^9 + 7x^8 + 22x^7 + 31x^6 + 52x^5 + 31x^4 + 22x^3 + 7x^2 + 3x + 1)$
    \bigstrut\\\cline{2-3}
    & $4$ & $(x + 1)(x^2 + x + 1)(x^3 + x^2 + x + 1)(x^8 + 2x^7 + 3x^6 + 4x^5 + 10x^4 + 4x^3 + 3x^2 + 2x + 1)$
    \bigstrut\\\cline{2-3}
    & $5$ & $(x + 1)(x^2 + x + 1)(x^3 + x^2 + x + 1)(x^4 + x^3 + x^2 + x + 1)(x^5 + x^4 + x^3 + x^2 + x + 1)$ \bigstrut\\
\hline
\multirow{10}{*}{$7$} & $1$ & $x^6 + 120x^5 + 1191x^4 + 2416x^3 + 1191x^2 + 120x + 1$
	\bigstrut\\\cline{2-3}
    & $2$ & $(x + 1)(x^{10} + 5x^9 + 39x^8 + 218x^7 + 562x^6 + 870x^5 + 562x^4 + 218x^3 + 39x^2 + 5x + 1)$
    \bigstrut\\\cline{2-3}
    & \multirow{2}{*}{$3$} & $(x + 1)(x^2 + x + 1)(x^{12} + 4x^{11} + 10x^{10} + 38x^9 + 79x^8 $ \\
    & & $ + 166x^7 + 244x^6 + 166x^5 + 79x^4 + 38x^3 + 10x^2 + 4x + 1)$
    \bigstrut\\\cline{2-3}
    & \multirow{2}{*}{$4$} & $(x + 1)^4(x^2 + x + 1)$ \\
    & & $(x^{12} + x^{11} + 4x^{10} + 4x^9 + 21x^8 + 43x^6 + 21x^4 + 4x^3 + 4x^2 + x + 1)$
    \bigstrut\\\cline{2-3}
    & \multirow{2}{*}{$5$} & $(x + 1)(x^2 + x + 1)(x^3 + x^2 + x + 1)(x^4 + x^3 + x^2 + x + 1)$ \\
    & & $(x^{10} + 2x^9 + 3x^8 + 4x^7 + 5x^6 + 12x^5 + 5x^4 + 4x^3 + 3x^2 + 2x + 1)$
    \bigstrut\\\cline{2-3}
    & \multirow{2}{*}{$6$} & $(x + 1)(x^2 + x + 1)(x^3 + x^2 + x + 1)(x^4 + x^3 + x^2 + x + 1)$ \\
    & & $(x^5 + x^4 + x^3 + x^2 + x + 1)(x^6 + x^5 + x^4 + x^3 + x^2 + x + 1)$ \\
\hline
\end{tabular}
\end{center}
\label{table:leKstep}
\end{table}%
The degree of $J_{n+1,\leq k}(x)$ is given by $\frac{k(2n-k+1)}{2}$ since the maximum of $\inv_{\leq k}$ is achieved
by the reverse of the identity, and the number of $j$-step inversions in this permutation is $n+1-j$. Summing
this number from $1$ to $k$ gives the claimed degree.
We would like a more complete description of the distribution function of $J_{n,\leq k}(x)$, but we leave it for future work.


\section{Future work and connections with other work} \label{sec:certified-and-modulo}

\subsection{$k$-step inversion tops that are zero modulo $d$}
\label{section:topsModd}
Recall that given an inversion $(a,b)$ in a permutation, the letter $a$ is called an \emph{inversion top}.
Kitaev and Remmel~\cite{MR2240770,MR2336014} considered inversions where the inversion top is zero modulo $d$ for a particular integer $d$. We adapt this definition to our setting by defining $\modinv_{d,k}(\pi)$ to be the number of $k$-step inversions with an inversion
top that is zero modulo $d$.
Let
\[
L_{n,d,k}(x) = \sum_{\pi \in \symS_n} x^{\modinv_{d,k}(\pi)}
\]
be the corresponding distribution function.
\begin{proposition}
The leading coefficient of $L_{n,2,n-1}(x)$ is 
\[
\left\lfloor \frac{n}{2} \right\rfloor^2 (n-2)!.
\]
Thus
\[
\frac{L_{n,2,n-1}(x)}{(n-2)!} = \left\lfloor \frac{n}{2} \right\rfloor^2 x + n(n-1) - \left\lfloor \frac{n}{2} \right\rfloor^2.
\]
\end{proposition}

\begin{proof}
The formula for the leading coefficient is proved as follows: In order to have one
$(n-1)$-step inversion with an even inversion top, a permutation must start with an even
number and end in some smaller number. Thus we get the formula
\[
(n-2)! \sum_{j=1}^{\lfloor \frac{n}{2} \rfloor}(2j-1).
\]
Simplification yields the claimed formula.
\end{proof}

We now generalize this proposition.

\begin{proposition}
Let $n/2 < k < n$ and $1 < d \leq n$. The degree of $L_{n,d,k}(x)$ is $\ell = \min(n-k,\lfloor \frac{n}{d} \rfloor)$
and its leading coefficient equals

\[
(n-2\ell)! \ell! \binom{n-k}{\ell}
\sum_{1\leq i_1 < i_2 < \dotsb < i_\ell \leq \lfloor \frac{n}{d} \rfloor}
\prod_{j = 1}^\ell
(di_j-2j+1).
\]
\end{proposition}

\begin{proof}
Since $k > n/2$, there are $n-k$ locations that an inversion top could be, since there are $n-k$ many $k$-step runs of length two, with the remaining $k$ many $k$-step runs being of length 1.
There are $\lfloor \frac{n}{d}\rfloor$ possible values for inversion tops.
Thus $\ell = \min (n-k,\lfloor \frac{n}{d}\rfloor)$ is the maximum number of inversions possible with inversion tops modulo $d$.
Hence $\ell$ is the degree of $L_{n,d,k}(x)$.

The sum selects the values of the inversion tops, the $j\uth$ smallest inversion top being $di_j$.
The $j\uth$ factor of the product represents the remaining possible values that could be the bottom of the inversion with top $di_j$.  
The positions of these inversions are chosen among the $n-k$ possible locations pairs, which is why we multiply by the binomial coefficient.
The sum had arranged the tops in increasing order, and hence the coefficient of $\ell!$ counts the ways of rearranging the tops among their $\ell$ positions.
The coefficient of $(n-2\ell)!$ counts the ways of assigning the remaining values to the remaining positions. 
\end{proof}

Table~\ref{table:Lndk} contains some empirical data for the case $d = 2$.

{\tiny{
\begin{table}[h]
\caption{The distribution function of $\modinv_{2,k}$, $L_{n,2,k}(x)$.}
\begin{center}\tiny
\begin{tabular}{|c|c|l|}
\hline
$n$ & $k$ & $L_{n,2,k}(x)$\\
\hline
$1$ & $1$ & $1$ \\
\hline
\multirow{2}{*}{$2$} & $1$ & $y+1$ \\
    & $2$ & $2$ \\
\hline
\multirow{3}{*}{$3$} & $1$ & $2(y+2)$ \\
    & $2$ & $y+5$ \\
    & $3$ & $6$ \\
\hline
\multirow{4}{*}{$4$} & $1$ & $4(y^2+4y+1)$\\
    & $2$ & $2(y+1)(y+5)$ \\
    & $3$ & $8(y+2)$ \\
    & $4$ & $24$ \\
\hline
\multirow{5}{*}{$5$} & $1$ & $12(y^2+6y+3)$\\
    & $2$ & $6(y+1)(y+9)$ \\
    & $3$ & $2(y^2+22y+37)$ \\
    & $4$ & $24(y+4)$ \\
    & $5$ & $120$ \\
\hline
\multirow{6}{*}{$6$} & $1$ & $36(y+1)(y^2+8y+1)$\\
    & $2$ & $4(4y^3+55y^2+94y+27)$ \\
    & $3$ & $6(y+1)(y^2+22y+37)$ \\
    & $4$ & $4(y+5)(13y+17)$ \\
    & $5$ & $72(3y+7)$ \\
    & $6$ & $720$ \\
\hline
\multirow{7}{*}{$7$} & $1$ & $144(y^3+12y^2+18y+4)$\\
    & $2$ & $2(37y^3+615y^2+1359y+509)$ \\
    & $3$ & $4(7y^3+204y^2+651y+398)$ \\
    & $4$ & $6(y^3+75y^2+387y+377)$ \\
    & $5$ & $12(13y^2+154y+253)$ \\
    & $6$ & $360(3y+11)$ \\
    & $7$ & $5040$ \\
\hline
\end{tabular}
\end{center}
\label{table:Lndk}
\end{table}%
}}
We would like a more complete description of the polynomial $L_{n,2,k}(x)$, but we leave it for future work.

\subsection{Paths}
\label{section:ipcni}
Dukes and Reifergerste~\cite{MR2628782} showed that the left boundary sum of $\pi$, written $\lbsum(\pi)$, is the number
of inversions in $\pi$ added to the number certified non-inversions. A \emph{certified non-inversion}
is an occurrence of the pattern $132$ which is neither part of a $1432$ nor a $1342$ pattern. The mesh patterns defined by Br\"and\'en and Claesson~\cite{PBAC} can be used to give an alternative definition: A certified non-inversions is an occurrence of the mesh pattern
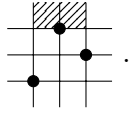
\[
\pattern{scale=1}{3}{1/1, 2/3, 3/2}{1/3,2/3}.
\]

In Dukes and Reifergerste~\cite{MR2628782}, the left boundary vector of a permutation $\pi$ of rank $n$ has as its $j\uth$ coordinate the largest $i<j$, such that $\pi_i > \pi_j$.  
The left boundary sum of a permutation $\pi$, denoted $\lbsum(\pi)$, is defined as the sum of the left boundary coordinates.

Variations of this may be as follows.
\begin{enumerate}
\item Define the $({\ge} k)$-left boundary vector of a permutation $\pi$ to be such that its $j\uth$ coordinate is the largest $i\le j-k$, such that $\pi_i > \pi_j$.
Define $\lbsum_{\ge k}(\pi)$ to be the the sum of this vector.
\item Define the $({\le} k)$-left boundary vector of a permutation $\pi$ to be such that its $j\uth$ coordinate is $i - \max (0,j-k)+1$, where $i$ is the largest, such that $\max(0,j-k) \le i < j$ and either $\pi_i > \pi_j$ or $i = \max(0,j-k)$.
Define $\lbsum_{\le k}(\pi)$ to be the sum of this vector.
\item Define the (${=}k$)-left boundary vector to be such that coordinate $j$ is $1$ if $a_{j-k} > a_j$, and $0$ otherwise.
Define $\lbsum_{=k}(\pi)$ to be the sum of this vector. 
\end{enumerate}

\begin{proposition}
\begin{enumerate}
\item Given a permutation $\pi$, $\lbsum_{\ge k}(\pi)$ is the number of (${\ge} k)$-step inversions, plus the number of non-inversions forming the end-points of an occurrence of the pattern $132$, but with at least $k$ steps from the $3$ to the $2$.  (Such a non-inversion consists of the endpoints of certified non-inversion within a permutation obtained by removing the $k-1$ positions to the left of the top of the non-inversion, and isomorphically adjusting the values to fit in the new range.) 
\item Given a permutation $\pi$, $\lbsum_{\le k}(\pi)$ is the number of $({\le} k)$-step inversions plus the number of certified $({\le} k)$-step non-inversions (a certified non-inversions whose endpoints are at most $k$ apart).
\item Given a permutation $\pi$, $\lbsum_{=k}(\pi)$ is the number of $k$-step inversions of $\pi$.
\end{enumerate}
\end{proposition}
\begin{proof}
The proof of these are adapted from \cite{MR2628782}.  
\begin{enumerate}
\item Let $(a_1,\ldots, a_n)$ be the $({\ge} k)$-left boundary vector.
Then $a_j = d$ is the maximum value no less than $k$ away from $j$, such that $\pi_d > \pi_j$.
Then every $i\le d$ is such that either $\pi_i > \pi_j$ or there is a position $c$ (we can always choose $d$), such that $i< c\le j-k$ and $\pi_c > \pi_j$.
Since $d$ is maximal, every $({\ge} k)$-step inversion with bottom $j$ will be counted among such $i$, and no $({<}k)$-step inversion will qualify, since $d$ is already $k$-steps away.
Furthermore, any $({\ge} k)$-step non-inversion $(i,j)$ with $i<d$, such that there is a $c$ (we can always pick $d$), such that $i < c\le j-k$ and $\pi_c >\pi_j$.
Because $d$ is maximal, any non-inversion $(i,j)$ that has a $c$, such that $i < c\le j-k$ and $\pi_c >\pi_j$, is a non-inversion, where $i<d$.
\item The proof here is almost identical to the previous case, except we restrict $d$ to ranging from $\max(0,j-k)$ to $j-1$, and we only consider $i \le d$, such that $i \ge j-k$.
This allows us to focus on $({\le} k)$-step inversions and certified $({\le} k)$-step non-inversions, and is accounted for by the fact that the $a_j$ is really $d-\max(0,j-k)+1$.
Furthermore, since we are technically counting certified non-inversions, which are triples rather than pairs, we select for the middle point the position of maximum value.  If we did not place this restriction, we could over-count, with many possibilities for a middle, given one pair of endpoints.   
\item The coordinates of the $({=}k)$-left boundary vector with the value $1$ are precisely the positions of the permutation that form the top of a $k$-step non-inversion.  Since a position can be the top of at most one $k$-step non-inversion, the sum of the coordinates of the vector is equal to the number of $k$-step non-inversions. 
\end{enumerate}
\end{proof}
We consider yet a forth way to generalize $\lbsum(\pi)$. For our purposes we define a \emph{certified $k$-step non-inversion} to be a certified non-inversion with endpoints forming a $k$-step non-inversion.
The generalization we focus on from here is as follows.
For a permutation $\pi$, let $\ipcni_k(\pi)$ be the number of $k$-step inversions in $\pi$ added to the
number of certified $k$-step non-inversions.
Let
\[
K_{n,k}(x) = \sum_{\pi \in \symS_n} x^{\ipcni_k(\pi)}
\]
be the corresponding distribution function.

\begin{table}[h]
\caption{The distribution function of $\ipcni_k$, $K_{n,k}(x)$.}
\begin{center}\tiny
\begin{tabular}{|c|c|l|}
\hline
$n$ & $k$ & $K_{n,k}(x)$ \bigstrut\\
\hline
$1$ & $1$ & $1$ \bigstrut\\
\hline
$2$ & $1$ & $x+1$ \bigstrut\\
    & $2$ & $2$ \\
\hline
$3$ & $1$ & $x^2+4x+1$ \bigstrut\\
    & $2$ & $2(2x+1)$ \\
    & $3$ & $6$ \\
\hline
$4$ & $1$ & $(x+1)(x^2+10x+1)$ \bigstrut\\
    & $2$ & $2(x+1)(5x+1)$ \\
    & $3$ & $6(3x+1)$ \\
    & $4$ & $24$ \\
\hline
$5$ & $1$ & $x^4+26x^3+66x^2+26x+1$ \bigstrut\\
    & $2$ & $2(13x^3+35x^2+11x+1)$ \\
    & $3$ & $6(11x^2+8x+1)$ \\
    & $4$ & $24(4x+1)$ \\
    & $5$ & $120$ \\
\hline
$6$ & $1$ & $(x+1)(x^4+56x^3+246x^2+56x+1)$ \bigstrut\\
    & $2$ & $2(38x^4+183x^3+121x^2+17x+1)$ \\
    & $3$ & $6(x+1)(46x^2+13x+1)$ \\
    & $4$ & $24(19x^2+10x+1)$ \\
    & $5$ & $120(5x+1)$ \\
    & $6$ & $720$ \\
\hline
$7$ & $1$ & $x^6+120x^5+1191x^4+2416x^3119x^2+120x+1$ \bigstrut\\
    & $2$ & $2(116x^5+969x^4+1100x^3+310x^2+24x+1)$ \\
    & $3$ & $6(202x^4+459x^3+157x^2+21x+1)$ \\
    & $4$ & $24(103x^3+89x^2+17x+1)$ \\
    & $5$ & $120(29x^2+12x+1)$ \\
    & $6$ & $720(6x+1)$ \\
    & $7$ & $5040$ \\
\hline
$8$ & $1$ & $(x+1)(x^6+246x^5+4047x^4+11572x^3+4047x^2+246x+1)$ \bigstrut\\
    & $2$ & $2(382x^6+5124x^5+9517x^4+4420x^3+684x^2+32x+1)$ \\
    & $3$ & $6(986x^5+3454x^4+1925x^3+325x^2+29x+1)$ \\
    & $4$ & $24(x+1)(614x^3+201x^2+24x+1)$ \\
    & $5$ & $120(190x^3+125x^2+20x+1)$ \\
    & $6$ & $720(41x^2+14x+1)$ \\
    & $7$ & $5040(7x+1)$ \\
    & $8$ & $40320$ \\
\hline
\end{tabular}
\end{center}
\label{table:Knk}
\end{table}%

From the empirical data in Table~\ref{table:Knk} it seems that the constant term in $K_{n,k}(x)$ is always
equal to $k!$. This is proven below.

\begin{proposition}\label{prop:constant:ipcnik}
The constant term in $K_{n,k}$ is $k!$. Furthermore, the permutations $\pi\in \symS_n$, such that $\ipcni_k(\pi) = 0$, are precisely the permutations that have the form $\sigma (k+1)(k+2) \dotsm n$,
where $\sigma\in \symS_k$.
\end{proposition}

\begin{proof}
It is clear that any permutation of the form $\sigma (k+1)(k+2) \dotsm n$, where $\sigma$ is a permutation from $S_k$ has $\ipcni_k$ zero. 
Conversely, suppose that $\ipcni_k(\pi) = 0$ and that $\pi$ is of the form $\sigma \lambda$ where $\sigma$
consists of the first $k$ letters of $\pi$ and $\lambda$ consists of the remaining letters.
Then the letters of $\lambda$ must
be in increasing order; otherwise we let $\ell_1 > \ell_2$ be the first two adjacent letters
in $\lambda$ that are not in increasing order. 
Then if $\pi_{\ell_2-k} < \pi_{\ell_2}$, there is an $\ell$, such that the triple $(\ell_2-k, \ell, \ell_2)$ is a certified $k$-step non-inversion, and if $\pi_{\ell_2-k} > \pi_{\ell_2}$, the pair $(\ell_2-k, \ell_2)$ is an inversion.
Now to finish the proof we need to show that $\sigma$
consists of the letters $1, \dotsc, k$. First observe that the first letter of $\lambda$ is larger
than any letter in $\sigma$.
Since $\lambda$ is increasing, the remaining letters in $\lambda$ are also larger than $1,\dotsc,k$. This finishes
the proof.
\end{proof}

\begin{corollary}
The number of permutations $\pi\in \symS_n$ with $\ipcni_{n-1}(\pi) = 1$ is $(n-1)(n-1)!$.
Thus
\[
\frac{K_{n,n-1}(x)}{(n-1)!} = (n-1)x+1.
\]
\end{corollary}

\begin{proof}
As a consequence of Proposition~\ref{prop:constant:ipcnik}, given $\pi\in \symS_n$, $\ipcni_{n-1}(\pi) =0$ if and only if $\pi_n = n$.
Note that $\ipcni_k(\pi) =1$ otherwise.
There are $(n-1)(n-1)!$ permutation $\pi$, such that $\pi_n\neq n$.
\end{proof}

\begin{proposition}
The number of permutations $\pi\in \symS_n$, with $\ipcni_{n-2}(\pi) = 2$ is
\[
(n-2)!(n^2-3n+1).
\]
Thus
\[
\frac{K_{n,n-2}(x)}{(n-2)!} = (n^2-3n+1)x^2+2(n-1)x+1.
\]
\end{proposition}

\begin{proof}
In order to construct a permutation with $\ipcni_{n-2}$ equal to $2$ we need to
choose four numbers to occupy the first two positions and the last two positions.
The permutation constructed in this way will always have $\ipcni_{n-2}$ equal to 2 \emph{unless} any of the following hold:
\begin{itemize}
\item $n$ is in position $1$ and $n-1$ is in position $n$,
\item $n$ is in position $n-1$, or
\item $n$ is in position $n$.
\end{itemize}
This shows that the number we are looking for is
\begin{align*}
(n-4)!\bigg( &n(n-1)(n-2)(n-3)(n-4) \\
&- (n-2)(n-3) - 2(n-1)(n-2)(n-3) \bigg).
\end{align*}
When this is simplified, it gives the formula in the proposition.
\end{proof}

\subsection{Marked mesh patterns}
Marked mesh patterns were defined by \'Ulfarsson in~\cite[Definition 24]{U11}. In this subsection we show how these patterns
relate to the concepts introduced above. Figure~\ref{fig:simplemarkedmesh} shows how $k$-step, $({\leq} k)$-step and $(k_1,k_2)$-step
inversions can be identified with patterns.
\begin{figure}[h]
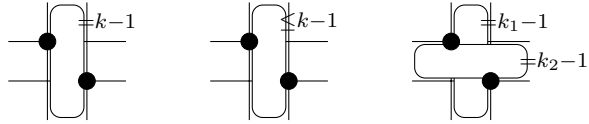

\begin{center}

\patternsbmm{scale=1.5}{ 2 }{ 1/2, 2/1 }{}{1/0/2/3/x}{2.5/2/2.5/3/{= k-1}} \quad
\patternsbmm{scale=1.5}{ 2 }{ 1/2, 2/1 }{}{1/0/2/3/x}{2.5/2/2.5/3/{\leq k-1}} \quad
\patternsbmm{scale=1.5}{ 2 }{ 1/2, 2/1 }{}{1/0/2/3/x, 0/1/3/2/x}{2.6/2/2.6/3/{=k_1-1}, 3.6/1/3.6/2/{=k_2-1}}

\caption{$k$-step, $({\leq} k)$-step and $(k_1,k_2)$-step inversions in terms of patterns.}
\label{fig:simplemarkedmesh}
\end{center}
\end{figure}

Using the representation of $k$-step inversions allows us to write the number of inversions and the inversion sum of a permutation as a linear combination of patterns; see Figure~\ref{fig:lincombmeshpatt}.
\begin{figure}[h]
\begin{center}

$\inv = \displaystyle\sum_{k \geq 1} \left( \patternsbmm{scale=1.5}{ 2 }{ 1/2, 2/1 }{}{1/0/2/3/x}{2.5/2/2.5/3/{= k-1}} \right), \quad
\invsum = \displaystyle\sum_{k \geq 1} k \cdot \left( \patternsbmm{scale=1.5}{ 2 }{ 1/2, 2/1 }{}{1/0/2/3/x}{2.5/2/2.5/3/{= k-1}} \right)
$

\caption{Writing the $\inv$ and $\invsum$ as a linear combination of patterns.}
\label{fig:lincombmeshpatt}
\end{center}
\end{figure}

It is only slightly harder to realize that the coordinates of the zone-crossing vectors are given by patterns, for example, the $k\uth$ coordinate of the inversion zone-crossing vector of a permutation $\pi$ is the number of occurrences of the pattern
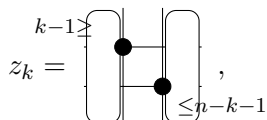
\[
z_k = \patternsbmm{scale=1.5}{ 2 }{ 1/2, 2/1 }{}{0/0/1/3/x, 2/0/3/3/x}{-0.5/2/-0.5/3/{k-1 \geq}, 3.5/0/3.5/1/{\leq n-k-1}},
\]
in $\pi$.
A $k$-step inversion with an inversion top that is zero modulo $d$ is an occurrence of
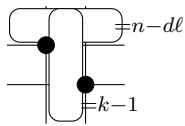
\[
\patternsbmm{scale=1.5}{ 2 }{ 1/2, 2/1 }{}{0/2/3/3/x, 1/0/2/3/x}{2.6/0/2.6/1/{=k-1}, 3.6/2/3.6/3/{=n-d\ell}}
\]
for some $\ell \geq 1$.

Finally, a certified $k$-step non-inversion is an occurrence of the pattern
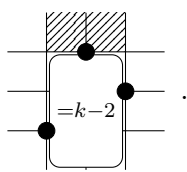
\[
\patternsbmm{scale=1.5}{ 3 }{ 1/1, 2/3, 3/2 }{1/3,2/3}{1/0/3/3/x}{1/0/3/3/{=k-2}}.
\]

\subsection{Acknowledgements}
We would like to thank Anders Claesson and Einar Steingr\'imsson for their helpful comments.

\bibliographystyle{amsplain}
\bibliography{Refs-k-step-invs}

\end{document}